\documentclass[a4paper,10pt]{article}

\usepackage{amsmath,amssymb,amsfonts}
\usepackage{ifthen}
\usepackage{hyperref}
\usepackage[amsmath,thmmarks,hyperref]{ntheorem}
\usepackage{graphicx}

\theoremstyle{plain}
\theoremseparator{.}
\newtheorem{theorem}{Theorem}
\newtheorem{proposition}[theorem]{Proposition}
\newtheorem{lemma}[theorem]{Lemma}
\newtheorem{corollary}[theorem]{Corollary}

\theorembodyfont{}
\theoremsymbol{\ensuremath{{}_\blacksquare}}

\newtheorem{remark}[theorem]{Remark}

\theoremsymbol{\ensuremath{\square}}
\theoremheaderfont{\itshape}
\newtheorem*{proof}{Proof}

\newcommand{\te}{{\tilde{\eta}}}
\newcommand{\ve}{\varepsilon}
\newcommand{\req}[1]{\eqref{eq:#1}}

\newcommand{\R}{\field{R}}
\newcommand{\N}{\field{N}}

\newcommand{\Lp}[1]{{L_\diamond^{#1}(\Omega)}}

\newcommand{\abs}[1]{\left| #1 \right|}

\newcommand{\dist}[2]{{\text{dist}\left(#1,#2\right)}}

\newcommand\set[1]{\left\{ #1 \right\}}

\newcommand{\norm}[1]{\left\| #1 \right\|}

\newcommand{\Reg}[2]{{\cal R}_{#1}^{#2}}
\newcommand{\Fun}[2]{{\cal F}_{#1}^{#2}}

\DeclareMathOperator{\sgn}{sgn}

\newcommand{\fd}{{f^\delta}}

\newcommand{\field}[1]{\ensuremath{\mathbb{#1}}}

\newcommand{\NextScriptStyle}[1]{{\scriptstyle{#1}}}
\newcommand{\NextScriptScriptStyle}[1]{{\scriptscriptstyle{#1}}}
\newcommand{\NextTextStyle}[1]{{\textstyle{#1}}}
\newcommand{\NextDisplayStyle}[1]{{\displaystyle{#1}}}
\newcommand{\SwitchBracketsizeLeft}[1]{
  \ifthenelse{\equal{#1}{b}\OR\equal{#1}{big}}{\let\LeftBracketSize=\bigl}{
    \ifthenelse{\equal{#1}{B}\OR\equal{#1}{Big}}{\let\LeftBracketSize=\Bigl}{
      \ifthenelse{\equal{#1}{g}\OR\equal{#1}{bigg}}{\let\LeftBracketSize=\biggl}{
     \ifthenelse{\equal{#1}{G}\OR\equal{#1}{Bigg}}{\let\LeftBracketSize=\Biggl}{
      \ifthenelse{\equal{#1}{s}\OR\equal{#1}{small}}{\let\LeftBracketSize=\NextScriptStyle}{
        \ifthenelse{\equal{#1}{ss}}{\let\LeftBracketSize=\NextScriptScriptStyle}{
          \ifthenelse{\equal{#1}{t}\OR\equal{#1}{text}}{\let\LeftBracketSize=\NextTextStyle}{
         \ifthenelse{\equal{#1}{d}\OR\equal{#1}{display}}{\let\LeftBracketSize=\NextDisplayStyle}{
          \ifthenelse{\equal{#1}{a}\OR\equal{#1}{auto}}{\let\LeftBracketSize=\left}{
            \let\LeftBracketSize=\relax}}}}}}}}}}
\newcommand{\SwitchBracketsizeRight}[1]{
  \ifthenelse{\equal{#1}{b}\OR\equal{#1}{big}}{\let\RightBracketSize=\bigr}{
    \ifthenelse{\equal{#1}{B}\OR\equal{#1}{Big}}{\let\RightBracketSize=\Bigr}{
      \ifthenelse{\equal{#1}{g}\OR\equal{#1}{bigg}}{\let\RightBracketSize=\biggr}{
     \ifthenelse{\equal{#1}{G}\OR\equal{#1}{Bigg}}{\let\RightBracketSize=\Biggr}{
      \ifthenelse{\equal{#1}{s}\OR\equal{#1}{small}}{\let\RightBracketSize=\NextScriptStyle}{
        \ifthenelse{\equal{#1}{ss}}{\let\RightBracketSize=\NextScriptScriptStyle}{
          \ifthenelse{\equal{#1}{t}\OR\equal{#1}{text}}{\let\RightBracketSize=\NextTextStyle}{
         \ifthenelse{\equal{#1}{d}\OR\equal{#1}{display}}{\let\RightBracketSize=\NextDisplayStyle}{
          \ifthenelse{\equal{#1}{a}\OR\equal{#1}{auto}}{\let\RightBracketSize=\right}{
            \let\RightBracketSize=\relax}}}}}}}}}}

\makeatletter
\newcommand{\logmessage}[1]{\@latex@warning{#1}}

\newcommand{\ignore}{\logmessage{Text ignored}\@gobble}
\makeatother

\title{A Derivative Free Approach for Total Variation Regularization}
\author{Carsten Pontow${}^2$ \quad Otmar Scherzer${}^{1,3}$\\
{ \footnotesize
\begin{tabular}{ccc}
\hbox to 0pt{\hss ${}^1$}Computational Science Center &
\hbox to 0pt{\hss ${}^2$}Department of Mathematics &
\hbox to 0pt{\hss ${}^3$}RICAM \\
University of Vienna & University of Innsbruck & Radon Institute\\
Nordbergstrasse~15 & Technikerstrasse~21a &  Altenbergerstrasse~69\\
1090 Wien, Austria & 6020 Innsbruck, Austria &  4040 Linz, Austria\\
\end{tabular}} }

\begin{document}

\maketitle

\begin{abstract}
The goal of this paper is to present a novel approach for total variation regularization and Sobolev minimization, which
are prominent tools for variational imaging. Thereby we use derivative free characterizations of the total variation semi-norm
and Sobolev semi-norms of functions recently derived by Bourgain, Br\'ezis, Mironescu and D\'avila.
Their analysis is to approximate the semi-norms of a function by singular integral operators. With this characterization we derive a
series of novel regularization methods for total variation minimization which have as a novel feature a non-local
double integral regularization term.
\end{abstract}

\section{Introduction}
\medskip
Given noisy image data $\fd$, \emph{total variation denoising} (see \cite{RudOshFat92}) consists in minimization of the functional
\begin{equation*}
   \Fun{}{1}(f) :=  \frac{1}{2} \int_\Omega (f-\fd)^2(x)\,dx + \alpha |Df|\;.
\end{equation*}
The minimizer is a smoothed approximation of $\fd$.
In the above functional $|Df|$ denotes the total variation seminorm of $f$ and $\alpha >0$ is a positive constant.
The first summand of the functional above is called the \emph{fidelity} term and penalizes the deviation of an image $f$ from the data $\fd$. The second summand is named regularization term and penalizes the rate of change within $f$. While the first summand provides that the outcome of the minimization process (the denoised image) preserves similarity to $\fd$, the second term is intended to reduce the oscillations within the argument $f$ in order to generate an approximation to $\fd$ that is free of the inherent noise. In fact, this strategy has proven to be successful and even more, the total variation seminorm has proven to be superior to other regularization terms penalizing the rate of change in the sense that edges within the image are preserved and not blurred.

Another well-known choice for the regularization term is for $1 < p < \infty$ the $p$-th power of the Sobolev $(1,p)$-seminorm $|\cdot|_{1,p}$  given by
\[
   |f|_{1,p} := \left(\int_\Omega |\nabla f(x)|^p \; dx\right)^\frac{1}{p}.
\]
 Exchanging it with the TV-seminorm in the above functional yields the functional
\begin{equation*}
   \Fun{}{p}(f) :=  \frac{1}{2} \int_\Omega (f-\fd)^2(x)\,dx + \alpha |f|^p_{1,p}.
\end{equation*}

In this paper we are concerned with a analytical approach to approximate the functionals $\Fun{}{p}$ in a variational sense.
Recently, new {\bf derivative free characterizations} of the Sobolev spaces $W^{1,p}$ and the space $BV$ of functions of bounded total variation function have been obtained by Bourgain, Br\'ezis and Mironescu \cite{BouBreMir00} and D\'avila \cite{Dav02} -- This work has been refined and supplemented by Ponce in \cite{Pon04b}. These characterizations provided new derivative-free approximations of the $(1,p)$- and total variation seminorms that are obtained by approximating the respective seminorm of a function $f$ by double integrals over the $p$-th power of the difference quotient function of $f$ multiplied with weighting kernel functions that form an approximation the identity.

In detail, let for $1 \leq p <\infty$ and any measurable function $f$
\begin{equation}
\label{eq:reg}
\Reg{n}{p}(f)  := \int_\Omega \int_\Omega \frac{\abs{f(x) - f(y)}^p}{\abs{x-y}^p} \varphi_n(x-y)\,dx\,dy\,.
\end{equation}
The functions $\varphi_n$  are   non-negative, radially symmetric, and radially decreasing functions from $L^1(\R^N)$ satisfying that for every $\delta > 0$
\begin{equation}
\label{eq:varphi1}
 \lim_{n \to \infty} \int_{\set{x : \abs{x} > \delta}} \varphi_n(x)\,dx = 0
\end{equation}
and for all $n \in \N$
\begin{equation}
\label{eq:varphi2}
\int_\Omega \varphi_n(x)\,dx  =1\,.
\end{equation}
Conditions (\ref{eq:varphi1}) and (\ref{eq:varphi2}) imply that the unit mass of the functions $\varphi_n$ concentrates around the origin as $n$ strives to infinity.

If $1 < p <\infty$, then by \cite[Theorem 2]{BouBreMir00} there exist real constants $K_{p,N}$ such that
for every measurable function $f$ defined on $\Omega$
\begin{equation*}
\frac{1}{K_{p,N}} \lim_{n \to \infty} \Reg{n}{p} (f) =
 \int_\Omega \abs{\nabla f(x)}^p\,dx  = \abs{f}_{1,p}^p\,,
\end{equation*}
where the limit is $+\infty$ if $f$ does not belong to $W^{1,p}(\Omega)$.
For the space of functions with finite total variation a similar characterization with $p=1$
in the integral expression holds true \cite{Dav02}:
\begin{equation*}
 \frac{1}{K_{1,N}}   \lim_{n \to \infty} \Reg{n}{1} (f) = \abs{Df}.
\end{equation*}

These approximations of the seminorms give rise to the following approximations of the functionals $\Fun{}{p}$ from above. Let for $1 \leq p < \infty$
\begin{equation}
\label{eq:fn}
  \Fun{n}{p}(f) := \frac{1}{2} \int_\Omega (f -\fd)^2(x)\,dx + \frac{\alpha}{K_{p,N}} \Reg{n}{p} (f).
\end{equation}

In this paper we give a variational analysis of the approximation  of the functionals $\Fun{}{p}$ by the functionals $\Fun{n}{p}$ for $n \to \infty$. In particular, we  will show that all of the approximating functionals $\Fun{n}{p}$ have a unique minimizer $f_n$ and that the sequence of minimizers $(f_n)$ has an accumulation point $f$ that is indeed the unique minimizer of the limit functional $\Fun{}{p}$. That is, the minimizers of $\Fun{n}{p}$ approximate the minimizer of $\Fun{}{p}$.
Most of these results are applications of results found in \cite{BouBreMir00}, \cite{Dav02} and \cite{Pon04b}.

Further, we show how the regularization functional $\Reg{n}{1}$ evaluates for some important examples. In this way, taking into account the above mentioned
result of the paper, we derive a series of numerical schemes for total variation minimization. In fact well-known numerical methods for total variation minimization can be derived, as well as a series of new ones.
In future work this approach could perhaps be used to clarify the relationships between existing and seemingly unrelated numerical
and analytical approaches or to supply an analytical basis for existing numerical schemes.
A particular consequence of our approach is that by the used approximation it turns out that total variation minimization can be
considered a \emph{bilateral filtering} \cite{TomMan98}.
There is still an ongoing discussion on comparing qualities of total variation denoising and bilateral filtering numerically (see e.g. \cite{PizDidBauWei07}). The results of this paper shed some additional light on this topic from an analytical point of view.
Another issue is that, in contrast to total variation minimization, the new functionals do not require the differentiability of the total variation
measure, and thus the derivation of numerical schemes can be considered in a function setting.

\bigskip \textbf{Notations.}
We summarize some further general assumptions and notations that will be used throughout the rest of this paper.
The set $\Omega \subseteq \R^N$ is a bounded open set with $C^1$-boundary.
The symbol $\mathcal{H}^{N-1}$ denotes the $(N-1)$-dimensional Hausdorff-measure in $\R^N$. The symbol $\mathcal{L}$ denotes the $N$-dimensional Lebesgue measure in $\R^N$.

For a real function $f$ as usual $f_+ := \max(f,0)$ and $f_- := -\min(f,0)$ symbolize the positive and negative parts of $f$, respectively.

The letter $p$ is always used as integration index and always satisfies $1 \leq p < \infty$. Sometimes the range of $p$ will be further restricted. The letter $p^*$ denotes the dual index to $p$ and is given by
\[
 \frac{1}{p} + \frac{1}{p^*} = 1
\]
if $p > 1$ while $p^* = \infty$ if $p = 1$. We denote by $\norm{f}_p$ the $L^p$-norm of $f$ on $\Omega$.
The space of $L^p$-functions on $\Omega$ with mean value zero is symbolized by
\[
 L^p_\diamond(\Omega) :=  \left\{f \in L^p(\Omega) \; : \; \int_\Omega f(x) \; dx = 0 \right\}.
\]
The mean value of an integrable function $f$ on $\Omega$ is symbolized by
\[
 f_\Omega := \frac{1}{\mathcal{L}(\Omega)} \int_\Omega f(x) \; dx.
\]

The characteristic function of a set $A \subset \R^N$ is denoted by $\chi_A$.

Let $C_c^\infty(\Omega)$ be the space of infinitely differentiable functions from $\Omega$ to $\R$ with compact support. For the $TV$-seminorm of a locally integrable function $f$ we use the symbol
 \[
  |Df| = \sup\left\{\int_\Omega f(x) \nabla \cdot \psi(x) \; dx \; : \; \psi \in C_c^\infty(\Omega;\R^n), |\psi(x)| \leq 1 \; \mathrm{for} \; \mathrm {all} \; x \in \Omega\right\}.
 \]
The space of functions of bounded variation on $\Omega$ is the set
\[
  BV(\Omega) := \{f \in L^1(\Omega) \; : \; |Df| < \infty\}.
\]

The constants $K_{p,N}$ are defined by
\[
  K_{p,N} = \left\{ \begin{array}{cc}
 \frac{1}{{\cal{H}}^{N-1}(S^{N-1})} \int_{S^{N-1}}  \left| \left<e, \sigma\right> \right|^p \; d\mathcal{H}^{N-1}(\sigma) & \mathrm{if } \;  N > 1,\\
          &  \\
 1 & \mathrm{if } \; N = 1.
                \end{array} \right.
\]

The radial and radial decreasing functions $\varphi_n$ give rise to the monotone decreasing functions  $\tilde{\varphi}_n$ on $(0,\infty)$ defined for all $r >0$ by
\[
 \tilde{\varphi}_n(r) :=  \varphi_n(x)
\]
 where $x$ is a point  in $\R^N$ with $|x| = r$.

Let
\begin{eqnarray*}
 {\cal{S}}: L^1(\Omega) & \to &\R \cup \{+ \infty\}\\
           f &\mapsto &\frac{1}{2}\|f - \fd\|_2^2
\end{eqnarray*}
denote the first summand of the functional $\Fun{n}{p}$.
The functional attains the value $+\infty$ if $f-\fd$ not in $L^2$.

\section{Variational Analysis}
Let $1 \leq p < \infty$. In the following we prove existence and uniqueness of a minimizer of the functional $\Fun{n}{p}$ defined in \req{fn}.

We begin with the following Lemma.

\begin{lemma}
\label{le:regu}
 Let $f$ be a real measurable function that does not belong to $L^p(\Omega)$. Then
\[
\Reg{n}{p} (f) = +\infty.
\]
\end{lemma}
\begin{proof}

By the hypothesis
\[
   \int_\Omega |f(x)|^p \; dx = + \infty.
\]

By the assumptions on $\varphi_n$ there must exist a radius $r > 0$ with $\tilde{\varphi}(r) > 0$ and $\tilde{\varphi}(s) \geq \tilde{\varphi}(r)$ for all $s < r$.

Since $\Omega$ is bounded $\bar{\Omega}$ can be covered by finitely many balls of radius $r$, and there must exist a measurable subset $M$ of $\Omega$ with positive measure satisfying
\[
 \int_M |f(x)|^p \; dx = + \infty \quad \mathrm{and} \quad \mathrm{diam}(M) < r.
\]

 It follows that for any $\beta \in \R$ the function $f - \beta \notin L^p(M)$; in particular $f - f(x) \notin  L^p(M)$ for almost all $x \in M$.

Now we have
\begin{equation*}
 \begin{aligned}
\Reg{n}{p} (f) \geq &
\int_M \int_M \frac{|f(y) - f(x)|^p}{\abs{x-y}^p} \varphi_n(x-y) \; dy dx \\
\geq &
\frac{1}{r^p}\int_M \int_M {|f(y) - f(x)|^p} \varphi_n(x-y) \; dy dx
  \end{aligned}
\end{equation*}
and by monotonicity
\begin{equation*}
\begin{aligned}
 \frac{1}{r^p} \int_M \int_M {|f(y) - f(x)|^p} \varphi_n(x-y) \; dy dx
\geq &
\frac{\tilde{\varphi}_n(r)}{r^p} \int_M \int_M {|f(y) - f(x)|^p}\; dy dx. \\
\end{aligned}
\end{equation*}

But as inferred above for almost all $x \in M$
\[
 \int_{M} |f(y) - f(x)|^p \; dy  = +\infty.
\]
Thus,
\[
 \int_M \int_{M} |f(y) - f(x)|^p \; dy dx = +\infty,
\]
and by the above estimations the proof is complete.
\end{proof}

We continue with the following lemma about the weak lower semicontinuity of the regularization functional.

\begin{lemma}
\label{le:lsc}
Let $1 \leq p,q < \infty$. For all $n \in \N$, $\Reg{n}{p}$ is weakly lower semicontinuous on $L^q(\Omega)$, that is,
\begin{equation*}
 \Reg{n}{p}(g) \leq  \liminf_{k \to \infty} \Reg{n}{p}(g_k)
\end{equation*}
for every sequence $(g_k) \in L^q(\Omega)$ that converges weakly with respect to the $L^q$-topology to a function $g \in L^q(\Omega)$.
\end{lemma}
\begin{proof}
By a standard result of convex analysis (see e.g. \cite{Dac89}) it suffices to show that the functional $\Reg{n}{p}$ is convex and lower semicontinuous on $L^q(\Omega)$. Both properties of $\Reg{n}{p}$ are established below by representing the functional as the pointwise
supremum of convex and lower semicontinuous functionals on $L^q(\Omega)$.

Let $g \in L^q(\Omega)$.
Note that $\Reg{n}{p}(g)$ may be viewed as the $p$-th power of the $p$-norm on $\Omega \times \Omega$ of the following measurable function
\begin{equation}
\label{eq:hatf}
 \hat{g}(x,y) = \frac{g(y) - g(x)}{\abs{y-x}} \varphi_n^\frac{1}{p}(y-x)\;.
\end{equation}
Let $\Delta := \set{(x,x) \,:\, x \in \Omega}$ and
\begin{equation*}
\Pi_\ve = \set{(x,y) \in \Omega \times \Omega\,:\, \dist{(x,y)}{\Delta} > \ve}\;.
\end{equation*}

By monotone convergence
\begin{equation*}
 \norm{\hat{g}}_{L^p(\Omega \times \Omega)} = \sup_{\ve > 0} \set{\norm{\hat{g} \left. \right|_{\Pi_{\ve}} }_{L^p(\Pi_\ve)}}\;.
\end{equation*}

Let $\ve$ be sufficiently small such that $\Pi_\ve$ is not the empty set. We use the notation
\begin{equation*}
{\cal B}_{p^*}^\ve:=
\set{\te \in C_c(\Omega \times \Omega) \,:\, \,
     \left. \te \right|_{(\Omega \times \Omega) \setminus \Pi_\ve}  = 0 \text{ and }
     \norm{\te}_{L^{p^*}(\Pi_\ve)} \leq 1}
\end{equation*}
for the set of continuous functions on $\Omega \times \Omega$ with compact support and $L^{p^*}$-norm less than one that vanish in the complement of $\Pi_\ve$ in $\Omega \times \Omega$.

By duality, the $p$-norm of $\hat{g}$ on $\Pi_\ve$ can be represented as the following supremum:
\begin{equation*}
 \norm{\hat{g} \left. \right|_{\Pi_{\ve}} }_{L^p(\Pi_\ve)} =
 \sup \set{\int_{\Pi_\ve} \! \! \hat{g}(x,y) \eta(x,y) d(y,x) : \eta \! \in \! C_c(\Pi_\ve) \, \mathrm{and} \,
           \norm{\eta}_{L^{p_*}(\Pi_\ve)} \! \leq 1} \!.
\end{equation*}
Using that the functions from ${\cal B}_{p_*}^\ve$ vanish in $(\Omega \times \Omega) \setminus \Pi_\ve$ this supremum can be rewritten as
\[
 \sup_{\te \in {\cal B}_{p_*}^\ve} \int_\Omega \int_\Omega \hat{g}(x,y) \te(x,y)\,dy\,dx.
\]

 We need the following property of the functions $\te$ from the set ${\cal B}_{p_*}^\ve$: for all such $\te$ the function
\begin{equation}
\label{eqn101}
    x \mapsto \int_\Omega  \varphi_n^\frac{1}{p}(x-y) \frac{\te(x,y)}{\abs{x-y}}\,dy
\end{equation}
belongs to $L^{q}(\Omega)$. To see this note that for $p > 1$ we get by an application of H\"older's inequality and by using the vanishing property of $\te$ and  the hypothesis that  $\varphi_n$ has integral one
\begin{eqnarray*}
\begin{aligned}
\int_\Omega \left(   \int_\Omega \varphi_n^\frac{1}{p}(x-y) \frac{|\te(x,y)|}{\abs{y-x}}\,dy \right)^q \; dx
\leq &
\int_\Omega \left( \int_{\Omega \setminus B(x,\epsilon)} \varphi_n^\frac{1}{p}(x-y) \frac{\|\te\|_\infty}{\abs{y-x}}\,dy \right)^q dx \\
\leq &
\int_\Omega \left( \int_{\Omega \setminus B(x,\epsilon)} \varphi_n^\frac{1}{p}(x-y) \frac{\|\te\|_\infty}{\ve}\,dy \right)^q dx \\
\leq &
\left(\frac{\|\te\|_\infty}{\ve}\right)^q  {\mathcal{L}}(\Omega)^\frac{2q}{p^*} dx \\
< & \; \infty.
 \end{aligned}
\end{eqnarray*}
For $p = 1$ an analogous  computation can be carried out similarly.

In the same way it is proven that for all $\te \in {\cal B}_{p_*}^\ve$ the function
\begin{equation}
\label{eqn100}
  y \mapsto \int_\Omega  \varphi_n^\frac{1}{p}(x-y) \frac{\te(x,y)}{\abs{x-y}}\,dx
\end{equation}
 lies in $L^{q}(\Omega)$. By Fubini's theorem we can infer that
\begin{eqnarray*}
 \int_\Omega  \int_\Omega |g(y)| \varphi_n^\frac{1}{p}(x-y) \frac{|\te(x,y)|}{\abs{x-y}}\,dy\,dx\,
&  = &
 \int_\Omega   |g(y)|  \int_\Omega\varphi_n^\frac{1}{p}(x-y) \frac{|\te(x,y)|}{\abs{x-y}}\,dx\,dy\,\\
& < & \infty
\end{eqnarray*}
and thus the integral
\[
 \int_\Omega  \int_\Omega g(y) \varphi_n^\frac{1}{p}(x-y) \frac{\te(x,y)}{\abs{x-y}}\,dy\,dx
\]
converges.

Using the preceding results we can rewrite
\begin{equation*}
 \begin{aligned}
 \norm{\hat{g}}_{L^p(\Omega \times \Omega)}
= & \sup_{\ve > 0} \norm{\hat{g} \left. \right|_{\Pi_{\ve}} }_{L^p(\Pi_\ve)}
 \\
 = &
 \sup_{\ve > 0} \sup_{\te \in {\cal B}_{p_*}^\ve} \int_\Omega \int_\Omega \hat{g}(x,y) \te(x,y)\,dy\,dx\\
 = &
 \sup_{\ve > 0} \sup_{\te \in {\cal B}_{p_*}^\ve}
 \left( \int_\Omega  \int_\Omega g(y) \varphi_n^\frac{1}{p}(x-y) \frac{\te(x,y)}{\abs{x-y}}\,dy\,dx\,  - \right.\\
        & \qquad \qquad  \left. \int_\Omega g(x) \int_\Omega  \varphi_n^\frac{1}{p}(x-y) \frac{\te(x,y)}{\abs{x-y}}\,dy\,dx \right).
\end{aligned}
\end{equation*}
Exchanging the two variables in the first integral expression followed by an application of Fubini's theorem shows that the difference from above is equal to
\begin{equation*}
 \begin{aligned}
 &\sup_{\ve > 0} \sup_{\te \in {\cal B}_{p_*}^\ve}
 \left( \int_\Omega g(x) \int_\Omega  \varphi_n^\frac{1}{p}(x-y) \frac{\te(y,x)}{\abs{x-y}}\,dy\,dx\,  - \right.\\
        & \qquad \qquad  \left. \int_\Omega g(x) \int_\Omega  \varphi_n^\frac{1}{p}(x-y) \frac{\te(x,y)}{\abs{x-y}}\,dy\,dx \right)\\
 = &
 \sup_{\ve > 0} \sup_{\te \in {\cal B}_{p_*}^\ve}
 \int_\Omega g(x) \int_\Omega \varphi_n^\frac{1}{p}(x-y) \frac{\te(y,x) - \te(x,y)}{\abs{y-x}}\,dy\,dx\;.
 \end{aligned}
\end{equation*}
For every $\ve >0$ and any $\te \in {\cal B}_{p_*}^\ve$ the inner integral expression in the double integral expression above gives rise to the function
\begin{eqnarray*}
h_\te: \Omega & \to &   \R \\
      x  & \mapsto & \int_\Omega \varphi_n^\frac{1}{p}(x-y) \frac{\te(y,x) - \te(x,y)}{\abs{y-x}}\,dy.
 \end{eqnarray*}
Just as in (\ref{eqn101}) and (\ref{eqn100}) it is shown that all the functions $h_\te$ belong to $L^{q}(\Omega)$.

Thus, for all $\epsilon > 0$ the functional
\begin{equation*}
  g \mapsto \int_{\Omega} g(x) h_\te(x)\,dx
\end{equation*}
is continuous on $L^q(\Omega)$ for all  $\te  \in {\cal B}_{p_*}^\ve$, and since for all $g \in L^q(\Omega)$
\[
 \left({\Reg{n}{p}}\right)^\frac{1}{p}(g) = \norm{\hat{g}}_{L^p(\Omega \times \Omega)} = \sup_{\ve > 0} \sup_{\te \in {\cal B}_{p_*}^\ve} \int_{\Omega} g(x) h_\te(x)\,dx
\]
the  functional $\left({\Reg{n}{p}}\right)^\frac{1}{p}$ is the pointwise supremum of continuous functionals on $L^q(\Omega)$.
We conclude that the functional $\Reg{n}{p}$ is indeed the supremum of continuous functionals on $L^q(\Omega)$ and hence,
is lower semicontinuous on this space. The convexity of $\Reg{n}{p}$ can be proven in a likewise manner.
\end{proof}

Below we will use some compactness results of Bourgain, Br\'ezis, and Mironescu \cite{BouBreMir00} which are as follows.
\begin{theorem}
\label{th:brezis_lp}
Let $1 \leq p < \infty$. Assume that $(g_n)$ is a sequence of functions in $\Lp{p}$ such that $\Reg{n}{p}(g_n)$ is uniformly bounded.
Then the sequence $(g_n)$ is relatively compact in $\Lp{p}$ and has a subsequence $(g_{n_k})$ converging (in the $L^p$-norm)
to a limit function $g$ that lies in $W^{1,p}(\Omega)$ if $p >1$ and in $BV(\Omega)$ if $p=1$.
\end{theorem}

In the following we apply the above results to prove existence and uniqueness of a minimizer of the functional $\Fun{}{p}$.

Here we make use of a scale space property of variational denoising algorithms, that they are grey level invariant (see \cite{SchWei00}).
We note that for all $1 \leq p < \infty$ a function $f_0$ is a minimum of $\Fun{}{p}$ if and only if $f_0 - \int_\Omega \fd$ minimizes the modification of $\Fun{}{p}$ where $\fd$ is replaced by the function $\fd - \int_\Omega \fd$.
Since the latter function has mean value zero we restrict our attention to the case that the mean of $\fd$ is zero, and consequently,
also the mean of the minimizer of $\Fun{}{p}$ is zero.

\begin{proposition}
\label{co:exist2}
Let $1 \leq p < \infty$ and $\alpha > 0$ and assume that $\fd \in \Lp{2}$.
\begin{enumerate}
 \item
Then the functional $\Fun{n}{p}$ attains a unique minimizer $f_n$ over $\Lp{2}$ that also belongs to
$\Lp{p}$.

\item The function $f_n$ is also a minimizer of $\Fun{n}{p}$ over $L^1(\Omega)$.

\item The sequence of numbers $(\Reg{n}{p} (f_n))$ is uniformly bounded over $n \in \N$, and the sequence $f_n$ has a convergent subsequence whose limit $f$ is an element of $W^{1,p}(\Omega)$ if $p > 1$ and of $BV(\Omega)$ if $p = 1$.
\end{enumerate}
\end{proposition}
\begin{proof}
We begin with the proof of the first assertion and show first that the functional $\Fun{n}{p}$ attains a unique minimizer $f_n$ over $\Lp{2}$.
Let $(g_k)$ be a minimizing sequence of functions in $\Lp{2}$ for the functional $\Fun{n}{p}$ such that
\begin{equation*}
 \lim_{k \to \infty} \Fun{n}{p} (g_k) = \inf_{g \in \Lp{2}} \Fun{n}{p}(g) \leq \Fun{n}{p}(0) = \frac{1}{2} \norm{\fd}_2^2 < \infty\;.
\end{equation*}
Since
\[
 \|g_k - \fd\|_2^2 \leq 2\Fun{n}{p} (g_k)
\]
for all $k \in \N$
this implies that $(g_k)$ is uniformly bounded in $\Lp{2}$ and thus has a weakly convergent subsequence in $\Lp{2}$. Let us denote
this subsequence again by $(g_k)$ and its weak limit by $f_n \in \Lp{2}$.

We notice that the first summand ${\cal{S}}$ of $\Fun{n}{p}$ is convex and continuous on $\Lp{2}$ and thus, weakly lower semicontinuous on this space \cite{Dac89}.
From Lemma \ref{le:lsc} it follows that $\Reg{n}{p}$ and thus, also $\Fun{n}{p}$ are weakly lower semicontinuous on $\Lp{2}$, too.

It follows that
\[
    \Fun{n}{p}(f_n) \leq  \liminf_{k \to  \infty} \Fun{n}{p}(g_k) =  \inf_{g \in \Lp{2}} \Fun{n}{p}(g).
\]
Thus, $f_n$ is a minimizer for $\Fun{n}{p}$.

As the first summand ${\cal{S}}$ of $\Fun{n}{p}$ is the composition of a convex and a strictly increasing function it is strictly convex. By (the proof of) Lemma \ref{le:lsc} the second summand  $\Reg{n}{p}$ of $\Fun{n}{p}$ is convex and thus, we can infer strict convexity for the whole functional $\Fun{n}{p}$.
The latter implies the uniqueness of the minimizer.

By Lemma \ref{le:regu}  the minimizer $f_n$ also belongs to $\Lp{p}$.

 We now show the second assertion, namely, that $f_n$ minimizes $\Fun{n}{p}$ over $L^1(\Omega)$.

Since the mean value of $\fd$ is zero, for each $g  \in L^2(\Omega)$ we have
\begin{eqnarray*}
  \mathcal{S}\left(g - \int_\Omega g(y) \; dy \right)
&  = &  \int_\Omega \left( \left(g(x) - \int_\Omega g(y) \; dy \right) - \fd(x) \right)^2 \; dx \\
& = & \int_\Omega \left(g(x) - \fd(x)\right)^2  \; dx - \left( \int_\Omega g(x)  \; dx \right)^2 \\
& \leq &
\mathcal{S}(g)
\end{eqnarray*}
and
\[
\Reg{n}{p}\left(g - \int_\Omega g(y) \; dy \right)  = \Reg{n}{p}(g).
\]
Thus, for every $g \in L^2(\Omega)$ there exists a function $\tilde{g} \in \Lp{2}$ with
\[
\Fun{n}{p}(\tilde{g}) \leq \Fun{n}{p}(g),
\]
and it follows that the function $f_n$ is a minimizer of $\Fun{n}{p}$ over $L^2(\Omega)$. It is also a minimizer of $\Fun{n}{p}$ over $L^1(\Omega)$ since the first summand of $\Fun{n}{p}$ equals infinity for $g \in L^1(\Omega) \setminus L^2(\Omega)$.

For the proof of the third assertion note that
for all $n \in \N$
\[
 \Reg{n}{p}(f_n) \leq \frac{1}{\alpha}\Fun{n}{p}(f_n) \leq \frac{1}{\alpha}\Fun{n}{p}(0) = \frac{1}{2\alpha} \|\fd\|_2^2.
\]
Therefore, the sequence $(\Reg{n}{p}(f_n))$ is uniformly bounded. By theorem \ref{th:brezis_lp} the sequence $f_n$ has a convergent subsequence whose limit $f$ lies in $W^{1,p}(\Omega)$ if $p > 1$ and in $BV(\Omega)$ if $p = 1$.
\end{proof}

\textbf{Notation}: We denote the subsequence occurring in the third assertion of the previous proposition again by $(f_n)$ and use that notation for the rest of our paper. We denote by $f$ the limit of $(f_n)$ for the rest of the paper as well.

The following remark provides another justification that the mean of the minimizers $f_n$ is zero.
\begin{remark}
\label{remark}
 Let $h$ be the function
\begin{eqnarray*}
 h: \Omega & \to &\R \\
           x &\mapsto & 1.
\end{eqnarray*}
Note that for all $g \in L^2(\Omega)$ the functional $\Fun{n}{p}$ is Gateaux-differentiable in direction $h$ and that the Gateaux-derivative $({\Fun{n}{p}})'(g;h)$ satisfies
\[
 ({\Fun{n}{p}})'(g;h) = \int_\Omega (g - \fd)(x) \; dx
\]
since $\Reg{n}{p}(g+\delta h) = \Reg{n}{p}(g)$ for all $g \in L^2(\Omega)$ and all real $\delta \ne 0$.
Thus,  we can infer
\[
\int_\Omega (f_n - \fd)(x) \; dx =  ({\Fun{n}{p}})'(f_n;h) = 0
\]
which means that the mean value of the minimizer $f_n$  is equal to the mean value of $\fd$ that was assumed to be zero above.
\end{remark}

It remains to clarify whether the limit function $f$ is a minimum of the respective limit functional $\Fun{}{p}$. The concerning questions are answered to a large extent by a result of A. Ponce \cite{Pon04b} in terms of $\Gamma$-convergence.

We recall the definition of $\Gamma$-convergence in $L^1(\Omega)$ \cite{dal93}. Let $(F_n)$ denote a sequence of functionals mapping functions from $L^1(\Omega)$ to the set of extended real numbers $\bar{\mathbb{R}}$, and let $F$ be a functional of this kind, too. Then the sequence $(F_n)$ $\Gamma$-converges to $F$ with respect to the $L^1(\Omega)$-topology if and only if the following two conditions are satisfied
\begin{itemize}
 \item for every $g \in  L^1(\Omega)$ and for every sequence $(g_n)$ in $L^1(\Omega)$  converging to $g$ in the $L_1$-norm we have
\[
  F(g) \leq \liminf_{n \to \infty} F_n(g_n);
\]

\item for every $g \in  L^1(\Omega)$ there exists a sequence $(g_n)$ in $L^1(\Omega)$  converging to $g$ in the $L_1$-norm with
\[
  F(g) = \lim_{n \to \infty} F_n(g_n).
\]
\end{itemize}
In this case we write
\[
   \Gamma^-_{L^1(\Omega)}\textrm{-}\lim_{n \to \infty} F_n = F.
\]

We denote by sc$^-_{L^1(\Omega)} F$ the lower semicontinuous envelope of the functional $F$ with respect to the strong $L^1$-topology, that is, sc$^-_{L^1(\Omega)} F$ is the greatest lower semicontinuous functional less than or equal to $F$.

Ponce's result is established in a far more general setting than the one we are treating here. He investigates double integrals of the kind
\[
 \int_{\Omega} \int_{\Omega} \omega \left(\frac{|g(y) - g(x)|}{\abs{y-x}} \right) \rho_\epsilon (y-x) \; dy dx
\]
where $\omega$ is a continous function from $[0, \infty)$ to $[0, \infty)$ and $(\rho_\epsilon)_{\epsilon >0}$ is a family of nonnegative functions in $L^1(\mathbb{R}^N)$ satisfying the conditions (\ref{eq:varphi1}) and (\ref{eq:varphi2}). The functions $\rho_\epsilon$ are not assumed to be radial or radially decreasing.

The functions $\rho_\epsilon$ induce positive Radon measures $\mu_\epsilon$ on the sphere $S^{N-1}$: Let $B$ be a Borel subset of $S^{N-1}$ and let
\[
 \mathbb{R}_+ B := \{ rx \; : \; r \geq 0 \; \mathrm{and} \; x \in B\}.
\]
be the cone with its apex in the origin that is generated by $B$. Let
\[
  \mu_\epsilon(B) := \int_{\mathbb{R}_+ B} \rho_\epsilon(x) \; dx.
\]

The family of measures $(\mu_\epsilon)$ is bounded by $1$ and thus, has a subsequence $\mu_{\epsilon_j}$ that converges weakly to a Radon measure $\mu$ on $S^{N-1}$. Moreover,  let  $\omega_\mu$ be the real function on $\mathbb{R}^N$ defined by
\[
 \omega_\mu(\vec{v}) := \int_{S^{N-1}} \omega\left( \left| \left<\vec{v}, \sigma\right> \right| \right) \; d\mu(\sigma)
\]
for all $\vec{v} \in \mathbb{\R}^N$.

Further, let $\omega^{**}$ denote the convex lower semicontinuous envelope of the function $\omega$, that is, by our assumption, the greatest
convex function less than or equal to $\omega$.

Finally, Ponce defines the functional
\begin{eqnarray*}
 F: L^1(\Omega) & \to & [0, +\infty] \\
g & \mapsto & \left\{  \begin{array}{cc}
              \int_\Omega \omega_\mu(\nabla g(x)) \; dx & \mathrm{if} \; g \in C^1(\bar{\Omega}), \\
                +\infty      &  \mathrm{otherwise}.
             \end{array} \right.
\end{eqnarray*}

Ponce's result, which he proves for bounded open sets with Lipschitz boundary, is as follows:

\begin{theorem}
 If for all $x \in \mathbb{R}^N$
\[
 (\omega_\mu)^{**}(x)  = (\omega^{**})_\mu(x)
\]
then
\[
  \Gamma^-_{L^1(\Omega)}\textrm{-}\lim_{j \to \infty} \int_\Omega \int_\Omega \omega \left(\frac{|g(y) - g(x)|}{\abs{y-x}} \right) \rho_{\epsilon_j} (y-x) \; dy dx  = \mathrm{sc}^-_{L^1(\Omega)} F(g)
\]
for every $g \in L^1(\Omega)$.
\end{theorem}

Let us now apply Ponce's result to our problem.
Let $\mu_n$ denote the Radon measure induced by $\varphi_n$. Since $\varphi_n$ is radial with integral one we get
\[
   \mu_n(B) = \int_{\mathbb{R}_+ B} \varphi_n(x) \; dx =  \frac{{\cal{H}}^{N-1}(B)}{{\cal{H}}^{N-1}(S^{N-1})}
\]
for all Borel subsets $B$ of $S^{N-1}$.

Therefore, all the measures $\mu_n$ are equal and thus, in contrast to the more general situation treated by Ponce, the whole sequence $(\mu_n)$ is converging weakly to a limit measure $\mu$ that itself equals all of the measures $\mu_n$.

We note that in our case $\omega$ is the function defined by
\[
 \omega(x) := |x|^p
\]
for all $x \in \mathbb{R}^N$.
Thus, for all $\vec{v} \in \mathbb{R}^N$
\[
\omega_\mu(\vec{v}) = \int_{S^{N-1}}  \left| \left<\vec{v}, \sigma\right> \right|^p \; d\mu(\sigma) = \frac{1}{{\cal{H}}^{N-1}(S^{N-1})} \int_{S^{N-1}}  \left| \left<\vec{v}, \sigma\right> \right|^p \; d\mathcal{H}^{N-1}(\sigma) ,
\]
and since the integrand only depends on the length of $\vec{v}$ we have
\[
 \omega_\mu(\vec{v}) = \frac{|\vec{v}|^p}{{\cal{H}}^{N-1}(S^{N-1})} \int_{S^{N-1}}  \left| \left<e, \sigma\right> \right|^p \; d\mathcal{H}^{N-1}(\sigma) = K_{p,N} |\vec{v}|^p.
\]
where $e$ is an arbitrary unit vector in $\mathbb{R}^N$. Since both functions $\omega$ and $\omega_v$ are convex it follows that
\[
 (\omega_\mu)^{**}  = (\omega^{**})_\mu
\]
and thus, Ponce's theorem is applicable to our problem. We get
\[
  \Gamma^-_{L^1(\Omega)}\textrm{-}\lim_{n \to \infty} \int_\Omega \int_\Omega \frac{|g(y) - g(x)|^p}{\abs{y-x}^p}  \varphi_n (y-x) \; dy dx  = \mathrm{sc}^-_{L^1(\Omega)} F(g).
\]
for all $g \in L^1(\Omega)$.

We have to determine the functional  $\mathrm{sc}^-_{L^1(\Omega)} F$. From theorem 1.2. in \cite{CorDeA93} we get that for all $g \in L^1(\Omega)$
\[
  \mathrm{sc}^-_{L^1(\Omega)}F(g) = \left\{ \! \! \!  \begin{array}{cl}
                K_{p,N} \! \int_\Omega \left|D^a g(x)\right|^p \! \; dx + \int_\Omega  \omega_\mu^\infty \left(\frac{dD^sg}{d|D^sg|}  \right) \; d|D^sg|   & \mathrm{if} \; g \in BV(\Omega), \\
                +\infty      &  \mathrm{otherwise} \;
             \end{array} \right.
\]
where in the case $g \in BV(\Omega)$ the symbol $D^ag$ denotes the Radon-Nikodym derivative of the absolute continuous part of the vector-valued Radon measure $Dg$ with respect to Lebesgue measure and $\frac{dD^sg}{d|D^sg|}$ is the Radon-Nikodym derivative of the singular part $D^sg$ of $Dg$ with respect to its total variation $|D^sg|$. The function $\omega_\mu^\infty$ is the recession function: it is defined on $\Omega$, takes its values in the extended real numbers, and is  given by
\[
 \omega_\mu^ \infty(x) = \lim_{t \to \infty}  \frac{\omega_\mu(tx)}{t} =
 \left\{  \begin{array}{cc}
                K_{1,N} |x|  & \mathrm{if} \;  p = 1, \\
                +\infty      &  \mathrm{if} \;  p > 1.   \end{array} \right.
\]
for all $x \in \Omega$.

It follows that for $p = 1$
\[
  \mathrm{sc}^-_{L^1(\Omega)}F(f) = \left\{  \begin{array}{cc}
                K_{1,N}  |Df| & \mathrm{if} \; f \in BV(\Omega), \\
                +\infty      &  \mathrm{if} f \in L^1(\Omega) \setminus BV(\Omega).
             \end{array} \right.
\]
and for $p > 1$
\[
  \mathrm{sc}^-_{L^1(\Omega)}F(f) = \left\{  \begin{array}{cc}
                K_{p,N} |f|^p_{1,p} & \mathrm{if} \; f \in W^{1,p}(\Omega), \\
                +\infty      &  \mathrm{if} f \in L^1(\Omega) \setminus W^{1,p}(\Omega).
             \end{array} \right.
\]

Thus, the $\Gamma^-_{L^1(\Omega)}$-limit of the sequence of functionals $\Reg{n}{p}$ has been established for all $p \geq 1$. By \cite{dal93}, example 1.21., it is clear that the first summand of the functionals $\Fun{n}{p}$ is lower semicontinuous on $L^1(\Omega)$, and is thus $\Gamma^-_{L^1(\Omega)}$-converging to the first summand of the functionals $\Fun{}{p}$.  Altogether,
 the following theorem has been proved.
\begin{theorem}
 For all $1 \leq p < \infty$ the sequence of functionals $(\Fun{n}{p})$ converges in the $\Gamma_{L^1(\Omega)}^-$-sense to the limit functional $\Fun{}{p}$.
\end{theorem}

\begin{corollary}
 For all $1 \leq p < \infty$ the limit function $f$ of the (sub)sequence of minimizers $(f_n)$ of $\Fun{n}{p}$ is the unique minimum of the limit functional $\mathcal{F}_p$ over $L^1(\Omega)$. The minimum $f$ also belongs to the space $\L_\diamond^2(\Omega) \cap  W^{1,p}$ if $p > 1$ and $\Lp{2} \cap BV(\Omega)$ if $p = 1$.
\end{corollary}
\begin{proof}
 The limit function $f$ is a minimum of the limit functional $\mathcal{F}_p$ over $L^1(\Omega)$ due to the properties of $\Gamma^-$-convergence \cite{dal93}. It belongs to the space $\L^2(\Omega)$ since  $\mathcal{S}(f)$ is finite. Its mean is zero since it is the limit of the sequence $(f_n)$ from the closed subspace $\Lp{p}$ of $L^p(\Omega)$, and it belongs to $W^{1,p}$ if $p > 1$ and $BV(\Omega)$ if $p=1$ by the third assertion of Proposition \ref{co:exist2}.
\end{proof}
\section{Numerical Minimization of the Energy Functionals}

In this section we present some numerical schemes for minimiziation of the regularization functionals $\Reg{n}{1}$
in space dimensions one and two. We use a finite element approach and approximate functions by a linear combination of finite
elements, in particular piecewise constant and piecewise linear functions. Then, by tuning the kernel functions $\varphi_n$ we
are able to recover standard finite difference schemes for total variation minimization on the one hand and on the other hand
novel discrete schemes. A numerical comparison of the derived methods and applications to imaging will be studied in a forthcoming paper.
The numerical schemes are derived from an approximation of the total variation functional in an infinite dimensional setting, which are then 
discretized. These schemes can serve as alternatives to existing numerical schemes (see e.g. \cite{ChaLio97,AcaVog94,DobSch96,Vog02,JalCha09})
which are based on direct minimization of the Rudin-Osher-Fatemi functional and not on dual formulations, like the Chambolle's algorithm; 
see e.g. \cite{Cha04,AujCha05}.

\subsection{The One-dimensional Case}

We work on the domain
\[
  \Omega := (0,1).
\]
and consider minimization of $\Reg{n}{1}$ with a finite element method.

\begin{enumerate}
\item The first two schemes are for piecewise constant finite elements:
\begin{enumerate}
\item We use the sequence of kernel functions $(\varphi_n)$ defined by
\[
 \varphi_n := \frac{n}{2} \chi_{[-\frac{1}{n}, \frac{1}{n}]}.
\]
Let $a_1,\ldots,a_n \in \R$.
Evaluating the one-dimensional piecewise constant function
\[
 f_n := \sum_{i = 1}^n a_i \chi_{[\frac{i-1}{n}, \frac{i}{n}]}
\]
with $\Reg{n}{1}$ yields the standard $TV$-seminorm of $f_n$:
\[
 \Reg{n}{1}(f_n) = \sum_{i=2}^n |a_i - a_{i-1}| = |Df_n|.
\]
\item Using instead of $(\varphi_n)$ the family of kernels $(\varphi_n^{(2)})$ defined by
\[
 \varphi_n^{(2)} := \frac{n}{4} \chi_{[-\frac{2}{n}, \frac{2}{n}]}
\]
yields
\[
 \Reg{n}{1}(f_n) = \sum_{i=2}^{n-1} \frac{1-\ln(2)}{2} |a_{i+1} - a_{i-1}| + \sum_{i=2}^n \ln(2) |a_i - a_{i-1}|.
\]
We recall that $\ln(2) \approx 0.7$.
\end{enumerate}
\item Now, we consider a finite element method for piecewise linear splines. Let $a_0,\ldots, a_n \in \R$ and $f_n$ be the piecewise linear spline interpolating the nodes $(\frac{k}{n}, a_k), k = 0 \ldots n$, i.e.,
\[
 f_n := \sum_{i = 0}^n a_i  g_i
\]
where
\[
  g_i(x)  := \max\left(1 - n \abs{x - \frac{i}{n}},0\right).
\]

Inserting $f_n$ in $\Reg{n}{1}$ and using the kernel functions $\varphi_n$ yields
\[
 \Reg{n}{1}(f_n) = \sum_{i = 1}^n \frac{\abs{a_i - a_{i-1}}}{2} + \sum_{i = 1}^{n-1} t(a_{i-1},a_i,a_{i+1})
\]
where
\[
 t(a_{i-1},a_i,a_{i+1}) = \left\{ \begin{array}{cc}
\frac{|a_{i+1} - a_{i-1}|}{4} & \mathrm{if } \;  \sgn(a_{i-1} - a_i) = \sgn(a_i -  a_{i+1}),\\
          &  \\
 \frac{(a_{i} - a_{i-1})^2 + (a_i - a_{i+1})^2}{4(\abs{a_i - a_{i-1}} + \abs{a_i - a_{i+1}})} & \mathrm{if } \; \sgn(a_{i-1} - a_i) \ne \sgn(a_i -  a_{i+1}).

\end{array} \right.
\]
 We used the computer algebra program MAPLE for this evaluation. We provide a sketch of some parts of the  computation for $n \geq 2$. We want to evaluate
\[
 \int_0^1 \int_0^1 \frac{\abs{\sum_{i = 0}^n a_i  g_i(x) - \sum_{i = 0}^n a_i  g_i(y)}}{\abs{x-y}} \varphi_n(x-y) \; dx dy.
\]
which is equal to
\[
 \frac{n}{2} \sum_{k = 1}^n \int_\frac{k-1}{n}^\frac{k}{n} \int_{\max(0, y-\frac{1}{n})}^{\min(1, y+ \frac{1}{n})} \frac{\abs{\sum_{i = 0}^n a_i  g_i(x) - \sum_{i = 0}^n a_i  g_i(y) }}{\abs{x-y}} \; dx dy.
\]
Looking at the supports of the functions $g_i$ we realize that the latter double integral equals
\[
 \frac{n}{2} \left( \int_0^\frac{1}{n} \int_{0}^{y+ \frac{1}{n}} \frac{ \abs{\sum_{i = 0}^{2} a_i g_i(x)  -  \sum_{j=0}^1 a_j g_j(y)}}{\abs{x-y}} \; dx dy \right. +
\]
\[
  \sum_{k=2}^{n-1}  \int_{\frac{k-1}{n}}^\frac{k}{n} \int_{y-\frac{1}{n}}^{y+ \frac{1}{n}} \frac{\abs{\sum_{i = k-2}^{k+1} a_i g_i(x)  -  \sum_{j=k-1}^k a_j g_j(y) }}{\abs{x-y}} \; dx dy +
\]
\[
\left.  \int_{\frac{n-1}{n}}^1 \int_{y - \frac{1}{n}}^1 \frac{ \abs{\sum_{i = n-2}^{n} a_i g_i(x)  -  \sum_{j=n-1}^n a_j g_j(y)}}{\abs{x-y}} \; dx dy \right).
\]

We only treat the second double integral, the other two are evaluated analogously. Let $2 \leq k \leq n-1$.

Then
\[
\int_{\frac{k-1}{n}}^\frac{k}{n} \int_{y-\frac{1}{n}}^{y+ \frac{1}{n}} \frac{\abs{\sum_{i = k-2}^{k+1} a_i g_i(x)  -  \sum_{j=k-1}^k a_j g_j(y) }}{\abs{x-y}} \; dx dy
\]
can be decomposed into the sum
\[
 \int_{\frac{k-1}{n}}^\frac{k}{n} \int_{y-\frac{1}{n}}^{\frac{k-1}{n}} \frac{\abs{\sum_{i = k-2}^{k-1} a_i g_i(x)  -  \sum_{j=k-1}^k a_j g_j(y) }}{y-x} \; dx dy +
\]
\[
 \int_{\frac{k-1}{n}}^\frac{k}{n} \int_{\frac{k-1}{n}}^{\frac{k}{n}} \frac{\abs{\sum_{i = k-1}^{k} a_i g_i(x)  -  \sum_{j=k-1}^k a_j g_j(y) }}{\abs{x-y}} \; dx dy +
\]
\[
 \int_{\frac{k-1}{n}}^\frac{k}{n} \int_{\frac{k}{n}}^{y +\frac{1}{n}} \frac{\abs{\sum_{i = k}^{k+1} a_i g_i(x)  -  \sum_{j=k-1}^k a_j g_j(y) }}{x-y} \; dx dy.
\]
whose summands we denote by $I_{k,-}, I_{k}$ and $I_{k,+}$, respectively.
The integrand of $I_{k}$ is of the simple form 
\[
 I_{k} = {\frac {\abs{a_{{k}}-a_{{k-1}}}}{n}}.
\]

By an application of Fubini's theorem
\[
 I_{k,-} = I_{k-1,+}.
\]
 Thus, it suffices to treat the evaluation of $I_{k,-}$ whose integrand $J(x,y)$ is reshaped as follows:
\[\frac{\abs{
 \left(\! a_{{k-1}}-a_
{{k}} \right) \!n y \!+ \!\left( a_{{k-1}}-a_{{k-2}} \right) \!n x \!+ \!\left( a_{{k-2}}+a_{{k}}-2\,a_{{k-1}} \right) \! k \!  - \! a_{{k-2}} \! + \! 2\,
a_{{k-1}} \! - \! a_{{k}}}}{y-x}.
\]

From this representation it is already visible that the evaluation will be dependant from the sign of the differences
\[
 \Delta_k := a_k - a_{{k-1}} \quad  \mathrm{and} \quad \Delta_{k-1} = a_{{k-1}}-a_{{k-2}}.
\]
We treat here one instance of the more complex case when
\[
 \sgn(\Delta_k) \ne \sgn(\Delta_{k-1}),
\]
namely the subcase where the middle coefficient  $a_{k-1}$ is the maximum of the three coefficients; the other subcase where $a_{k-1}$ is the minimum can be treated just the same.
The less complex case where $a_{k-1}$ lies between $a_{k-2}$ and $a_k$ needs fewer case distinctions but apart from that can be treated analogously.
Note that in the chosen subcase the second difference
$
 \Delta^2_k = a_{{k-2}}-2\,a_{{k-1}} +a_{{k}}
$ is negative.

The numerator $J_N(x,y)$ of $J(x,y)$ now reads
\[
 J_N(x,y) = \abs{-\Delta_k n y + \Delta_{k-1} n x +  \Delta^2_k  (k -1)}
\]
and is positive if and only if
\[
  x > f_N(y) := \frac{\Delta_k n y  -  \Delta^2_k  (k -1)}{\Delta_{k-1} n}.
\]
Thus, to evaluate the inner integral of $I_{k,-}$ we have to determine the intersections
of its integration domain  $(y-\frac{1}{n},\frac{k-1}{n})$ with the intervals
$(f_N(y),+\infty)$ and $(-\infty,f_N(y))$, respectively, for all $y \in (\frac{k-1}{n}, \frac{k}{n})$, which is the integration domain of the outer integral.

We get that
\[
 f_N(y)  < \frac{k-1}{n} \quad \Longleftrightarrow \quad  \frac{k-1}{n} < y
\]
and
\[
 y- \frac{1}{n} <  f_N(y) \quad \Longleftrightarrow \quad y < \frac{1}{n}\left(k  -  \frac{\Delta_k}{\Delta^2_k} \right).
\]

Let $C := \frac{1}{n}\left(k  -  \frac{\Delta_k}{\Delta^2_k}\right)$. Note that $C < \frac{k}{n}$. Then
\[
 I_{k,-} = - \int_\frac{k-1}{n}^C \int_{ y- \frac{1}{n}}^{f_N(y)} J(x,y) \; dx dy + \int_\frac{k-1}{n}^C \int_{f_N(y)}^\frac{k-1}{n} J(x,y) \; dx dy +
\]
\[
 \int_C^\frac{k}{n} \int_{ y- \frac{1}{n}}^{\frac{k-1}{n}} J(x,y) \; dx dy.
\]

Integrating $J(x,y)$ with respect to $x$ yields the primitive function
\[
 K(x,y) := \left((ny - k + 1 ) \Delta^2_k\right)\ln(y-x) - x n \Delta_{k-1}.
\]

Inserting the limits of the inner integral of the third summand yields
\[
 K\left(y-\frac{1}{n}, y\right) = \left((k - 1 - ny) \Delta^2_k\right)\ln(n) - (ny - 1)\Delta_{k-1},
\]
\[
 K\left(\frac{k-1}{n},y\right) := \left((ny - k + 1 ) \Delta^2_k\right)\ln\left(y- \frac{k-1}{n}\right) - (k-1) \Delta_{k-1}
\]
and
\[
L_3(y) \! := \! \int_{ y- \frac{1}{n}}^{\frac{k-1}{n}} \! \! J(x,y) \; dx = \left((ny - k + 1) \Delta^2_k\right)\ln(ny-k+1) \! +\! (ny - k) \Delta_{k-1}.
\]

A primitive function for $L_3$  is
\[
 M_3(y) \!:= \! {\frac { \left( ny \!- \! k + 1 \right)^{2} \! \Delta^2_k }{2n}}\!\ln(n y - k + 1) - \! \frac{1}{4n} \!\left(n y - k + 1\right)^2 \!\Delta^2_k + \frac{1}{2}y(nk - 2y) \Delta_{k-1}.
\]

Analogous computations for the first and second summand yield the functions
\[
 L_2(y)  = \left( ny+1-k \right)\left(\Delta_k^2 \ln\left(- \frac{\Delta_{k-1}}{\Delta^2_k} \right)  +\Delta_k \right)
\]
and
\[
 M_2(y) = \frac{1}{2}\,{y}( ny -2k +2)\left( \Delta_k^2 \ln  \left(
-\frac{\Delta_{k-1}}{\Delta_k^2} \right) +
\Delta_k \right)
\]
and
\[
L_1(y) = \left(-\Delta_k^2 \right)  \left( ny-k+1 \right)
 \left( \ln  \left(-\frac{\Delta_{k}^2  (ny-k+1)}{\Delta_{k-1}} \right)  -1
 \right) +\Delta_{k-1}
\]
and
\[
M_1(y) = -\,\frac{\Delta_k^2}{2}   \left( y \left( -
ny-2+2\,k \right) \ln  \left( -\frac{\Delta_{k-1}}{\Delta_k^2} \right) \! + \! {\frac { \left( ny-k+1 \right) ^{2}
\ln  \left( ny-k+1 \right) }{n}}\right.
\]
\[ \left.
-{\frac { \left( ny-k+1 \right) ^
{2}+2\,ny \left( ny-2\,k \right) }{2n}} \right) + \left( a_{{k}}-a_{{k-
1}} \right) y,
 \]
respectively.

Now
\begin{eqnarray*}
 I_{k,-} & = & M_1\left(\frac{k-1}{n}\right) - M_1(C) + M_2(C) - M_2\left(\frac{k-1}{n}\right) + M_3\left(\frac{k}{n}\right) - M_3(C) \\
& = &  {\frac { \left( a_{{k-1}}-a_{{k}} \right) ^{2}+ \left( a_{{k-1}}-
a_{{k-2}} \right) ^{2}}{4n{\Delta_{{k}}}^{2}}}.
\end{eqnarray*}

From this result and the results from above the final result follows easily.

We further evaluated $\Reg{n}{1}$ for the Haar-functions $h_j^{(k)}$ using the kernel functions $(\varphi_n)$. For $j = k = 0$ the function $h_j^{(k)}$ is defined by
\[
  h_0^{(0)}(x) := 1
\]
for all $x \in \Omega$. For $k \in \N_{0}, 1\leq j \leq 2^k$ we have
\[
  h_j^{(k)}(x) := \left\{ \begin{array}{cc}
\sqrt{2^k} & \mathrm{if } \;  x \in \left(\frac{2j-2}{2^{k+1}},\frac{2j-1}{2^{k+1}}\right),\\
          &  \\
 -\sqrt{2^k}& \mathrm{if } \;  x \in \left(\frac{2j-1}{2^{k+1}},\frac{2j}{2^{k+1}}\right), \\
& \\
0 & \mathrm{otherwise}. \end{array} \right.
\]
Since $h_0^{(0)}$ is constant
\[
 \Reg{n}{1}\left(h_0^{(0)}\right) = 0.
\]

For $h_1^{(0)}$ we get
\[
  \Reg{n}{1}\left(h_0^{(1)}\right) = \left\{ \begin{array}{cc}
2 \ln(2)& \mathrm{if } \;  n = 1,\\
          &  \\
 2 & \mathrm{if } \;  n > 1.\end{array} \right.
\]

Note that by symmetry for $k \geq 1$ and $1 \leq j \leq 2^{k-1}$
\begin{equation}
\label{haar}
 \Reg{n}{1}\left(h_j^{(k)}\right) = \Reg{n}{1}\left(h_{2^k-j +1}^{(k)}\right).
\end{equation}

For $ k \geq 1 $ the marginal functions $h_1^{(k)}$  evaluate to
\[
 \Reg{n}{1}\left(h_1^{(k)}\right) = \left\{ \begin{array}{ll}
 \sqrt{2^k} \left( \left(k+ \frac{1}{2^{k-1}}\right) \ln(2)  - \left(1 - \frac{1}{2^{k}}  \right) \ln\left(2^k -1\right)\right) & \mathrm{if } \;  n = 1,\\
          &  \\
 \frac{n}{\sqrt{2^k}}\left( \left(k+2\right)\ln(2) - \ln(n) +1\right) & \mathrm{if } \;  2 \leq n \leq 2^k, \\
 & \\
  \sqrt{2^k} n  \left(\frac{k+1}{2^{k-1}} \ln(2) - \frac{1}{2^{k-1}} \ln(n) + \frac{1}{2^{k-1}} - \frac{1}{n}\right) & \mathrm{if }  \,  2^k \! \leq n \! \leq 2^{k+1}\!,  \\
 & \\
 3 \sqrt{2^k} & \mathrm{if } \;   n \geq 2^{k+1}.\end{array} \right.
\]

For $k \geq 2$ and $j = 2,\ldots,2^{k-1}$ the inner functions $h_j^{(k)}$,  evaluate to

\[\Reg{n}{1}\left(h_j^{(k)}\right) = \left\{ \begin{array}{ll}
\sqrt {{2}^{k}} \left( {\frac {j\ln  \left( j \right) }{{2}^{k}}}-{
\frac { \left( j-1 \right) \ln  \left( j-1 \right) }{{2}^{k}}} - \right. & \\
 \left.  \left( 1-{\frac {j}{{2}^{k}}} \right) \ln  \left( {2}^{k}-j \right) + \right. & \\
 \left. \left( 1-{\frac {j-1}{{2}^{k}}} \right) \ln  \left( {2}^{k}-j+1
 \right) +{\frac {\ln  \left( 2 \right) }{{2}^{k-1}}} \right)  & \mathrm{if } \;  n = 1,
  \\
\\
\frac{n}{\sqrt {{2}^{k}}} \left( j\ln  \left( j \right) +
 \left( j-1 \right) \ln  \left( j-1 \right) + \right.& \\
  \left. \left( k+2 \right) \ln
 \left( 2 \right) -
\ln  \left( n \right)  +  1  \right) & \mathrm{if } \;  \frac{2^k}{2^k -j}
\leq n \leq \frac{2^k}{j}, \\
\\
n\sqrt {{2}^{k}} \left(   {\frac { \left( j-1 \right) \ln  \left( j-1
 \right) }{{2}^{k}}}       -{\frac { \left( j+1 \right) \ln  \left( n
 \right) }{{2}^{k}}}- \right.& \\
\left. {\frac { \left( kj+k+2 \right) \ln  \left( 2
 \right) }{{2}^{k}}} +{\frac {j+1}{{2}^{k}}}-\frac{1}{n} \right)
& \mathrm{if } \; \frac{2^k}{j}
\leq n \leq \frac{2^k}{j-1}, \\
\\
\frac{2n}{\sqrt{2^k}} \left( \left( k+1 \right)
\ln  \left( 2 \right)  -\ln  \left( n \right) +1  \right) & \mathrm{if } \; \frac{2^k}{j-1}
\leq n \leq {2^{k+1}}, \\
\\
4 \sqrt{2^k} & \mathrm{if } \; n \geq 2^{k+1}.
\end{array} \right.
\]
The evaluation of $h_j^{(k)}$ for $j = 2^{k-1}+1,\ldots, 2^k$ is reduced to the evaluations directly above via (\ref{haar}).

\end{enumerate}

\subsection{The Two-dimensional Case}

We now switch to the two-dimensional case and evaluate $\Reg{n}{1}$ for a piecewise constant function defined on a subset of $\R^2$. In detail, let  $\Omega$ be chosen as the open square $(0,1) \times (0,1)$ and $f$ be defined on $\Omega$ via
\[
 f(\vec{v})  := \sum_{i,j =1}^n a_{i,j}  \chi_{I_{i,j}}(\vec{v})
\]
for all $\vec{v} \in \Omega$ where $I_{k,l} := I_k \times I_l$ with $I_k := (\frac{k-1}{n}, \frac{k}{n})$ for all $k,l \in \{1,\ldots, n\}$.

\begin{enumerate}
\item
We choose the kernel functions
\[
 \varphi_n(\vec{v}) := \frac{n^2}{\pi} \chi_{B\left(0,\frac{1}{n}\right)}(\vec{v})
\]
for all $\vec{v}$ in $\R^2$ where $B\left(0,\frac{1}{n}\right)$ denotes the ball around the origin with radius $\frac{1}{n}$. The sequence $(\varphi_n)$ satisfies all conditions stated in the introduction. We note that for two points $(x,y)$ and $(w,z)$ from $\Omega$ we have $(x,y) - (w,z) \in B\left(0,\frac{1}{n}\right)$ if and only if
\[
  (w,z) \in S_\Omega\left(x,y,\frac{1}{n}\right) := \left((0,1) \times (0,1)\right) \cap B\left((x,y), \frac{1}{n}\right).
\]
We further define the intersection of the circle $B((x,y),\frac{1}{n})$ with the square $I_{k,l}$ by
\[
 S_{k,l}\left(x,y,\frac{1}{n}\right) := I_{k,l} \cap B\left((x,y), \frac{1}{n}\right).
\]

We have to evaluate
\[
 \Reg{n}{1}(f) = \int_\Omega \int_\Omega \frac{\abs{f(x,y) - f(w,z)}}{\abs{(x,y) - (w,z)}} \varphi((x,y) - (w,z)) \;  d(w,z) \, d(x,y)
\]
\[
= \frac{n^2}{\pi} \int_0^1 \! \int_0^1 \! \int \!\!\!\int_{S_\Omega\left(x,y,\frac{1}{n}\right)} \!\! \!\! \!\!\!\!
\frac{\abs{ \sum_{i,j =1}^n a_{i,j}  \chi_{I_{i,j}}(x,y)- \sum_{k,l =1}^n a_{k,l}  \chi_{I_{k,l}}(w,z)}}{\abs{(x- w, y -z)}} \, d(w,z) \; dy  dx.
\]
The occurring quadruple integral can be rewritten as follows:
\[
 \sum_{i,j =1}^n \sum_{k,l =1}^n \int_{\frac{i-1}{n}}^\frac{i}{n} \! \int_{\frac{j-1}{n}}^\frac{j}{n} \! \int \!\!\!\int_{S_{k,l}\left(x,y,\frac{1}{n}\right)} \!\!
\frac{\abs{  a_{i,j}  -  a_{k,l}}}{\abs{(x- w, y -z)}}  \; d(w,z)\, dy  dx.
\]

For fixed $1 \leq i,j \leq n$ and fixed $x,y \in I_{i,j}$ the domain $S_{k,l}\left(x,y,\frac{1}{n}\right)$ of the inner double integral is empty if $\abs{i-k} \geq 2$ or $\abs{j-l} \geq 2$. Thus,  it suffices to evaluate those summmands of the inner sum above that fulfill  $\abs{i-k} \leq 1$ and $\abs{j-l} \leq 1$. However, in the case where $i = k$ and $ j = l$ the integrand of the corresponding summand vanishes such that this case may be left out, too.

Given a pair of indices $(i,j)$ let
\[
\mathcal{I}_{i ,j} := \{(k,l) \in \{1,\ldots,n\}^2 \; :\; \abs{i-k} \leq 1, \abs{j-l} \leq 1, (k,l) \ne (i,j) \}.
\]
denote the set of pairs of indices for which the corresponding summands of the inner sum in the integral expression above do  not vanish generally.

Then the above quadruple integral equals
\[
 \sum_{i,j =1}^n  \sum_{(k,l) \in \mathcal{I}_{i,j}}\abs{  a_{i,j}  -  a_{k,l}} \int_{\frac{i-1}{n}}^\frac{i}{n} \! \int_{\frac{j-1}{n}}^\frac{j}{n} \! \int \!\!\!\int_{S_{k,l}\left(x,y,\frac{1}{n}\right)} \!\!
\frac{1}{\abs{(x- w, y -z)}}  \; d(w,z)\, dy  dx.
\]
We denote for all $1 \leq i,j \leq n$ and all $(k,l) \in \mathcal{I}_{i,j}$  the quadruple integral expression on the right hand side above with $J_{i,j}^{k,l}$, i.e.
\[
 J_{i,j}^{k,l} := \int_{\frac{i-1}{n}}^\frac{i}{n} \! \int_{\frac{j-1}{n}}^\frac{j}{n} \! \int \!\!\!\int_{S_{k,l}\left(x,y,\frac{1}{n}\right)} \!\!
\frac{1}{\abs{(x- w, y -z)}}  \; d(w,z)\, dy  dx.
\]

Again, let $1 \leq i,j \leq n$ be fixed. The set of pairs of indices with non-vanishing summands $\mathcal{I}_{i ,j}$ may be partitioned into the sets
\[
\mathcal{I}_{i ,j}^d := \{(k,l) \in \mathcal{I}_{i ,j} \; : \; k \ne i \; \mathrm{and} \; l \ne j \}
\]
of pairs of indices marking squares diagonally adjacent to the square $I_{i,j}$ and
\[
 \mathcal{I}_{i ,j}^l := \{(k,l) \in \mathcal{I}_{i ,j} \; : \; k = i \; \mathrm{or} \; l = j \}
\]
collecting pairs of indices that denote squares laterally adjacent to $I_{i,j}$. By simple transformations of the kind $(x,w) \mapsto (x \pm \frac{1}{n}, w \pm \frac{1}{n})$ etc. and applications of Fubini's theorem (as in the one-dimensional case) (or by geometric insight) we realize that the integrals $J_{i,j}^{k,l}$ are equal for all $(k,l) \in \mathcal{I}_{i ,j}^d$ and the same holds true for all $(k,l) \in \mathcal{I}_{i ,j}^l$. Further, the respective values of the two evaluations  are independent of $i$ and $j$.

Thus, it suffices to compute the values of $J_{i_*,j_*}^{i_*-1,j_*-1}$ and $J_{i_*,j_*}^{i_*,j_*-1}$ for some fixed $2 \leq i_*,j_* \leq n$, and the final result will be
\[
 \Reg{n}{1}(f) = \frac{2 n^2}{\pi} \left( \! \left(\sum_{i=2}^n\sum_{j=2}^n\left|a_{i,j} - a_{i-1,j-1}\right| + \sum_{i=1}^{n-1} \sum_{j =2}^{n} \left|a_{i,j} - a_{i+1,j-1}\right| \!\right) \! \!J_{i_*,j_*}^{i_*-1,j_*-1} + \right.
\]
\begin{equation}
\label{result}
\left. \left( \sum_{i=1}^n\sum_{j=2}^n \left|a_{i,j} - a_{i,j-1}\right| + \sum_{i=2}^{n} \sum_{j =1}^{n} \left|a_{i,j} - a_{i-1,j}\right| \right) J_{i_*,j_*}^{i_*,j_*-1}\right).
\end{equation}

We begin with the more complex case of laterally adjacent squares and evaluate
\[
 J_{i,j}^{i,j-1} = \int_{\frac{i-1}{n}}^\frac{i}{n} \! \int_{\frac{j-1}{n}}^\frac{j}{n} \! \int \!\!\!\int_{S_{i,j-1}\left(x,y,\frac{1}{n}\right)} \!\!
\frac{1}{\abs{(x- w, y -z)}}  \; d(w,z)\, dy  dx.
\]
for some fixed $2 \leq i,j \leq n$.

We first point out that
\[
J_\frac{1}{2} := \int_{\frac{i-1}{n}+ \frac{1}{2n}}^\frac{i}{n} \! \int_{\frac{j-1}{n}}^\frac{j}{n} \! \int \!\!\!\int_{S_{i,j-1}\left(x,y,\frac{1}{n}\right)} \!\!
\frac{1}{\abs{(x- w, y -z)}}  \; d(w,z)\, dy  dx =
\]
\[ \int_{\frac{i-1}{n}}^{\frac{i-1}{n} + \frac{1}{2n}}\! \int_{\frac{j-1}{n}}^\frac{j}{n} \! \int \!\!\!\int_{S_{i,j-1}\left(x,y,\frac{1}{n}\right)} \!\!
\frac{1}{\abs{(x- w, y -z)}}  \; d(w,z)\, dy  dx
\]
such that
\[
 J_{i,j}^{i,j-1} = 2 J_\frac{1}{2}.
\]
This can be established by application of the transformations $(x,w) \mapsto (\frac{2i-1}{n}  -x, \frac{2i-1}{n} - w)$.

Let $(x,y)$ be chosen from $I_{i,j}$ with $x \geq \frac{i-1}{n}+ \frac{1}{2n}$. We analyze the inner double integral
\[
 \int \!\!\!\int_{S_{i,j-1}\left(x,y,\frac{1}{n}\right)} \!\!
\frac{1}{\abs{(x- w, y -z)}}  \; d(w,z)
\]
of $J_\frac{1}{2}$.
We use the abbreviations
\[
 a := x - \frac{i-1}{n}, \qquad   b  := y - \frac{j-1}{n} \qquad \mathrm{and}  \qquad d  :=  x - \frac{i}{n}.
\]
for the distances of $x$ and $y$ to some nodes. Note that by our choice of $x$ and $y$ we have the inequalities $a,b > 0$, $ d < 0$ and, in particular, $ a > |d|$.

Let $(w,z) \in S_{i,j-1}\left(x,y,\frac{1}{n}\right)$. It follows that
\[
  (x-w)^2 < \frac{1}{n^2} - (y-z)^2.
\]
and therefore,
\[
 z > y - \frac{1}{n}   \qquad \mathrm{and} \qquad  x - \sqrt{\frac{1}{n^2} - (y-z)^2} < w < x + \sqrt{\frac{1}{n^2} - (y-z)^2}.
\]
Thus, $z \in \left(y - \frac{1}{n}, \frac{j-1}{n}\right)$ and we have to analyze the intersection of intervals
\begin{equation}
\label{equiv1}
 I_w :=  \left(\frac{i-1}{n}, \frac{i}{n}\right) \cap \left(x - \sqrt{\frac{1}{n^2} - (y-z)^2},x + \sqrt{\frac{1}{n^2} - (y-z)^2}\right).
\end{equation}
(The index $w$ in  the symbol $I_w$ is just used as a symbol to indicate that we are dealing with the integration domain of the variable $w$ but does not stand for the values of $w$. The same applies to $I_z$ etc. below.)
We first point out that
\begin{equation}
\label{equiv3}
 \frac{i-1}{n} < x - \sqrt{\frac{1}{n^2} - (y-z)^2}  \; \Longleftrightarrow \; z < y - \sqrt{\frac{1}{n^2} - a^2}.
\end{equation}
Thus the result of (\ref{equiv1}) is dependent from the intersection
\begin{equation}
\label{equiv2}
I_z := \left(y - \frac{1}{n}, \frac{j-1}{n}\right) \cap \left(-\infty, y - \sqrt{\frac{1}{n^2} - a^2}\right).
\end{equation}
While it is clear that $y - \sqrt{\frac{1}{n^2} - a^2} > y - \frac{1}{n}$ we have
\begin{equation}
 \label{equiv5}
y - \sqrt{\frac{1}{n^2} - a^2} < \frac{j-1}{n} \; \Longleftrightarrow \; y < \frac{j-1}{n} + \sqrt{\frac{1}{n^2} - a^2}.
\end{equation}
with $\frac{j-1}{n} < \frac{j-1}{n} + \sqrt{\frac{1}{n^2} - a^2} < \frac{j}{n}$.

We first consider the case that $y \in I_y := \left(\frac{j-1}{n}, \frac{j-1}{n} + \sqrt{\frac{1}{n^2} - a^2}\right)$. Then $I_z = \left(y - \frac{1}{n},y - \sqrt{\frac{1}{n^2} - a^2}\right)$.

We look at the subcase $ z \in I_z$. It follows by (\ref{equiv3}) that now the lower bound of $I_w$ is $ x - \sqrt{\frac{1}{n^2} - (y-z)^2}$.

Considering its upper bound we have to find the minimum of \\ $x + \sqrt{\frac{1}{n^2} - (y-z)^2}$ and $\frac{i}{n}$. Similarly to (\ref{equiv3}) we get that
\begin{equation}
\label{equiv4}
 \frac{i}{n} > x + \sqrt{\frac{1}{n^2} - (y-z)^2}  \; \Longleftrightarrow \; z < y - \sqrt{\frac{1}{n^2} - d^2}.
\end{equation}
Since $a^2 > d^2$ it is true that $y - \sqrt{\frac{1}{n^2} - d^2} \in I_z$ implying that the treated subcase has two more subsubcases: $z \in I_{z,1} := \left(y - \frac{1}{n},y - \sqrt{\frac{1}{n^2} - d^2}\right)$ and $z \in I_{z,2} := \left(y - \sqrt{\frac{1}{n^2} - d^2},y - \sqrt{\frac{1}{n^2} - a^2} \right)$.

By (\ref{equiv4}) and (\ref{equiv1})   the upper bound of $I_w$ is $ x + \sqrt{\frac{1}{n^2} - (y-z)^2}$
if $z \in I_{z,1}$  and equals $ \frac{i}{n}$ if $z \in I_{z,2}$.
Thus, the treated subcase gives rise to the following two quadruple integrals
\[
 \int_{\frac{i-1}{n}+ \frac{1}{2n}}^\frac{i}{n} \! \int_{\frac{j-1}{n}}^{\frac{j-1}{n} + \sqrt{\frac{1}{n^2} - a^2}}\! \int_{y - \frac{1}{n}}^{ y - \sqrt{\frac{1}{n^2} - d^2}}\! \int_{ x - \sqrt{\frac{1}{n^2} - (y-z)^2}}^{x + \sqrt{\frac{1}{n^2} - (y-z)^2}}
\frac{1}{\abs{(x- w, y -z)}}  \; dw dz dy  dx
\]
\[
 \int_{\frac{i-1}{n}+ \frac{1}{2n}}^\frac{i}{n} \! \int_{\frac{j-1}{n}}^{\frac{j-1}{n} + \sqrt{\frac{1}{n^2} - a^2}}\! \int_{y - \sqrt{\frac{1}{n^2} - d^2}}^{ y - \sqrt{\frac{1}{n^2} - a^2}}\! \int_{ x - \sqrt{\frac{1}{n^2} - (y-z)^2}}^{ \frac{i}{n}}
\frac{1}{\abs{(x- w, y -z)}}  \; dw dz  dy  dx
\]
which we denote by $K_1$ and $K_2$, respectively.

We turn to the subcase that $z \notin I_z$, i.e. $z \in \left(y - \sqrt{\frac{1}{n^2} - a^2},\frac{j-1}{n}\right)$. By (\ref{equiv3}) the lower bound of $I_w$ is in this subcase $\frac{i-1}{n}$.
Since $y - \sqrt{\frac{1}{n^2} - a^2} > y - \sqrt{\frac{1}{n^2} - d^2}$ by (\ref{equiv4})  in this subcase the upper bound of $I_w$ is $\frac{i}{n}$. This subcase yields the quadruple integral
\[
K_3 :=  \int_{\frac{i-1}{n}+ \frac{1}{2n}}^\frac{i}{n} \! \int_{\frac{j-1}{n}}^{\frac{j-1}{n} + \sqrt{\frac{1}{n^2} - a^2}}\! \int_{y - \sqrt{\frac{1}{n^2} - a^2}}^{\frac{j-1}{n}}\! \int_{\frac{i-1}{n}}^{ \frac{i}{n}}
\frac{1}{\abs{(x- w, y -z)}}  \; dw dz  dy  dx.
\]

We still need to analyze the case $y \notin I_y$, that is, $y \in \left(\frac{j-1}{n} + \sqrt{\frac{1}{n^2} - a^2}, \frac{j}{n}\right)$. In this case by (\ref{equiv5}) and (\ref{equiv2}) the interval $I_z$ equals $\left(y - \frac{j-1}{n}, \frac{j-1}{n}\right)$.By (\ref{equiv5}) and (\ref{equiv2}) it is clear that the lower bound of $I_w$ is $x - \sqrt{\frac{1}{n^2} - (y-z)^2}$. The determination of the upper bound is a little more intricate including two subcases concerning the choice of the domain of $y$ one of which generating two subsubcases concerning the domain of $z$. However, its computation processes similarly enough to the computations in the first case that we skip it here and just state the resulting quadruple integrals which we name $K_4, K_5$ and $K_6$, respectively.
\[
 \int_{\frac{i-1}{n}+ \frac{1}{2n}}^\frac{i}{n} \! \int_{\frac{j-1}{n}+ \sqrt{\frac{1}{n^2} - a^2}}^{\frac{j-1}{n} + \sqrt{\frac{1}{n^2} - d^2}}\! \int_{y - \frac{1}{n}}^{ y - \sqrt{\frac{1}{n^2} - d^2}}\! \int_{ x - \sqrt{\frac{1}{n^2} - (y-z)^2}}^{x + \sqrt{\frac{1}{n^2} - (y-z)^2}}
\frac{1}{\abs{(x- w, y -z)}}  \; dw dz dy  dx,
\]
\[
 \int_{\frac{i-1}{n}+ \frac{1}{2n}}^\frac{i}{n} \! \int_{\frac{j-1}{n}+\sqrt{\frac{1}{n^2} - a^2}}^{\frac{j-1}{n} + \sqrt{\frac{1}{n^2} - d^2}}\! \int_{y - \sqrt{\frac{1}{n^2} - d^2}}^{ \frac{j-1}{n}}\! \int_{ x - \sqrt{\frac{1}{n^2} - (y-z)^2}}^{ \frac{i}{n}}
\frac{1}{\abs{(x- w, y -z)}}  \; dw dz  dy  dx,
\]
\[
  \int_{\frac{i-1}{n}+ \frac{1}{2n}}^\frac{i}{n} \! \int_{\frac{j-1}{n}+ \sqrt{\frac{1}{n^2} - d^2}}^{\frac{j}{n} }\! \int_{y - \frac{1}{n}}^{\frac{j-1}{n}}\! \int_{ x - \sqrt{\frac{1}{n^2} - (y-z)^2}}^{x + \sqrt{\frac{1}{n^2} - (y-z)^2}}
\frac{1}{\abs{(x- w, y -z)}}  \; dw dz  dy  dx.
\]
Altogether,
\[
  J_\frac{1}{2} = \sum_{i=1}^6 K_i.
\]

The evaluation of the six quadruple integrals $K_i$ involves the transformation of the respective  inner double integrals to polar coordinates. In order to simplify this procedure we first translate the integration domain of the respective inner double integral to the rectangle $(-\frac{1}{n}, \frac{1}{n}) \times (0, \frac{1}{n})$. In all six cases given a point $(x,y)$ from the domain of the respective outer double integral this is done by application of the transformation $(w,z) \mapsto ( x- w, y- z)$. Let $L_i$ be the result of this application to $K_i$. Then
\begin{eqnarray*}
 L_1 & = &
 \int_{\frac{i-1}{n}+ \frac{1}{2n}}^\frac{i}{n} \! \int_{\frac{j-1}{n}}^{\frac{j-1}{n} + \sqrt{\frac{1}{n^2} - a^2}}\!
\int^{\frac{1}{n}}_{\sqrt{\frac{1}{n^2} - d^2}}\! \int^{ \sqrt{\frac{1}{n^2} - z^2}}_{- \sqrt{\frac{1}{n^2} - z^2}}
\frac{1}{\abs{(w, z)}}  \; dw dz dy  dx,\\
L_2 &  = & \int_{\frac{i-1}{n}+ \frac{1}{2n}}^\frac{i}{n} \! \int_{\frac{j-1}{n}}^{\frac{j-1}{n} + \sqrt{\frac{1}{n^2} - a^2}}\!
\int^{\sqrt{\frac{1}{n^2} - d^2}}_{ \sqrt{\frac{1}{n^2} - a^2}}\! \int^{ \sqrt{\frac{1}{n^2} - z^2}}_{d}
\frac{1}{\abs{(w, z)}}  \; dw dz  dy  dx,\\
L_3 & = &  \int_{\frac{i-1}{n}+ \frac{1}{2n}}^\frac{i}{n} \! \int_{\frac{j-1}{n}}^{\frac{j-1}{n} + \sqrt{\frac{1}{n^2} - a^2}}\! \int^{\sqrt{\frac{1}{n^2} - a^2}}_{b}\! \int^{a}_{d}
\frac{1}{\abs{(w, z)}}  \; dw dz  dy  dx,\\
 L_4 & = & \int_{\frac{i-1}{n}+ \frac{1}{2n}}^\frac{i}{n} \! \int_{\frac{j-1}{n}+ \sqrt{\frac{1}{n^2} - a^2}}^{\frac{j-1}{n} + \sqrt{\frac{1}{n^2} - d^2}}\!
\int^{\frac{1}{n}}_{ \sqrt{\frac{1}{n^2} - d^2}}\! \int^{ \sqrt{\frac{1}{n^2} - z^2}}_{ -\sqrt{\frac{1}{n^2} - z^2}}
\frac{1}{\abs{(w, z)}}  \; dw dz dy  dx,\\
 L_5 & = & \int_{\frac{i-1}{n}+ \frac{1}{2n}}^\frac{i}{n} \! \int_{\frac{j-1}{n}+\sqrt{\frac{1}{n^2} - a^2}}^{\frac{j-1}{n} + \sqrt{\frac{1}{n^2} - d^2}}\!
\int^{\sqrt{\frac{1}{n^2} - d^2}}_{b}\! \int^{ \sqrt{\frac{1}{n^2} - z^2}}_{d}
\frac{1}{\abs{(w, z)}}  \; dw dz  dy  dx,\\
 L_6 & = & \int_{\frac{i-1}{n}+ \frac{1}{2n}}^\frac{i}{n} \! \int_{\frac{j-1}{n}+ \sqrt{\frac{1}{n^2} - d^2}}^{\frac{j}{n} }\! \int^{\frac{1}{n}}_{b}\! \int^{ \sqrt{\frac{1}{n^2} - z^2}}_{ - \sqrt{\frac{1}{n^2} - z^2}}
\frac{1}{\abs{(w, z)}}  \; dw dz  dy  dx.
\end{eqnarray*}

Let for $1 \leq i \leq 6$ the function $F_i$ be the evaluation function of the inner double integral of  $L_i$ defined on the domain of the outer double integral of $L_i$.

In $L_1, L_4$ and $L_6$  the integration domain of the inner double integral is a segment of the circle $B(0,\frac{1}{n})$ that results from the intersection of that circle with a parallel to the $x$-axis. A straight-forward transformation to polar coordinates $(r,\phi)$
yields
for example for $F_1$:
\begin{eqnarray*}
F_1(x,y)&  = &\int_{\sqrt{\frac{1}{n^2} - d^2}}^\frac{1}{n} \int_{\arcsin\left(\frac{\sqrt{\frac{1}{n^2} - d^2}}{r}\right)}^{\arcsin\left(-\frac{\sqrt{\frac{1}{n^2} - d^2}}{r}\right) + \pi} \; d\phi
dr \\
& = &
\int_{\sqrt{\frac{1}{n^2} - d^2}}^\frac{1}{n} \pi -2 \arcsin\left(-\frac{\sqrt{\frac{1}{n^2} - d^2}}{r}\right)  \;
dr \\
& = & \pi \,r
-2\,r\arcsin \left( {\frac {\sqrt {\frac{1}{n^2}-{d}^{2}}}{r}} \right) -\\
& & 2\,
\sqrt {\frac{1}{n^2}-{d}^{2}}\; { \mathrm{arcoth}} \left( {\frac {r}{\sqrt {{r}^{2}-
\frac{1}{n^2}+{d}^{2}}}} \right) \Bigg|_{\sqrt{\frac{1}{n^2} - d^2}}^\frac{1}{n} \\
& = & {\frac {\pi }{n}}
-2\,{\frac {\arcsin \left( \sqrt {\frac{1}{n^2} -{d}^{2}} n \right) }{n}}+\\
& & \sqrt{\frac{1}{{n}^{2}}-{d}^{2}} \left(\ln  \left(\frac{1}{n} +d  \right) -\ln  \left(\frac{1}{n} -d \right) \right),
\end{eqnarray*}
and $F_4$ and $F_6$ are treated analogously.

The integration domain of $F_3$ is a rectangle with edges parallel to the axes stretching across  both quadrants of the upper half plane. In order to transform this domain to polar coordinates we split it along the $y$-axis in two
axis-parallel rectangles that reside in the second and first quadrant, respectively,
\[
F_3(x,y) = \int^{\sqrt{\frac{1}{n^2} - a^2}}_{b}\! \int^{0}_{d}
\frac{1}{\abs{(w, z)}} \; dw dz + \int^{\sqrt{\frac{1}{n^2} - a^2}}_{b}\! \int^{a}_{0} \frac{1}{\abs{(w, z)}} \; dw dz,
\]
and call the resulting double integrals $A(x,y)$ and $B(x,y)$.

We turn to the computation of $A(x,y)$.
The transformation of a rectangle domain located in the second quadrant to polar coordinates depends on  whether its bottom left vertex or its top right vertex is more distant from the origin. In the case of $A(x,y)$ this conditions reads
\begin{equation}
 \label{equiv6}
 |(d,b)| < \left|\left(0,\frac{1}{n^2} - a^2\right)\right|.
\end{equation}

In the case where (\ref{equiv6}) holds true the transformation to polar coordinates yields
\[
  A(x,y) = \int_{b}^{|(d,b)|} \int_\frac{\pi}{2}^{-\arcsin\left(\frac{b}{r}\right) + \pi} \; d\phi dr + \int_{|(d,b)|}^{\left|\left(0,\frac{1}{n^2} - a^2\right)\right|}  \int_\frac{\pi}{2}^{\arccos\left(\frac{d}{r}\right)} \; d\phi dr +
\]
\[ \int_{\left|\left(0,\frac{1}{n^2} - a^2\right)\right|}^{\left|\left(d,\sqrt{\frac{1}{n^2}-a^2}\right)\right|} \int_{-\arcsin\left(\frac{\sqrt{\frac{1}{n^2} - a^2}   }{r} \right)+ \pi}^{\arccos\left(\frac{d}{r}\right)} \; d\phi dr,
\]
in the opposite case $A(x,y)$ equals
\[
 \int_{b}^{\left|\left(0,\frac{1}{n^2} - a^2\right)\right|} \int_\frac{\pi}{2}^{-\arcsin\left(\frac{b}{r}\right) + \pi} \; d\phi dr + \int_{\left|\left(0,\frac{1}{n^2} - a^2\right)\right|}^{|(d,b)|}  \int_{-\arcsin\left(\frac{\sqrt{\frac{1}{n^2} - a^2}}{r} \right)   + \pi}^{-\arcsin\left(\frac{b}{r}\right) + \pi} \; d\phi dr +
\]
\[
 \int_{|\left(d,b\right)|}^{\left|\left(d,\sqrt{\frac{1}{n^2}-a^2}\right)\right|} \int_{-\arcsin\left(\frac{\sqrt{\frac{1}{n^2} - a^2}   }{r} \right)+ \pi}^{\arccos\left(\frac{d}{r}\right)} \; d\phi dr.
\]

In both cases the occurring three double integrals can be evaluated similarly like $F_1(x,y)$ above. Summing together the respective three results  yields in both cases the same result:  a sum consisting of summands that are of one of the following three types:  binary products where one factor is  a logarithmic expression, binary products where one factor is an $\arcsin$- or $\arccos$-expression or binary products of a square root and $\pi$. By use of the appropriate transformation rules for arcus-expressions the two latter groups of binary products cancel each other out. Therefore, in both cases $A(x,y)$ equals
\begin{eqnarray*}
&&\frac{1}{2} \Bigg( b\left( \,\ln  \left(
\sqrt {{d}^{2}+{b}^{2}}+d \right) -\,\ln  \left( \sqrt {{d}^{2}+{b}^{2}}-d \right)  \right) + \\
& & d \Bigg(\,\ln  \left( \sqrt {{d}^{2}+{b}^{
2}}+b \right) -\,\ln  \left( \sqrt {{d}^{2}+{b}^{2}}-b \right)  + \\
& & \left. \,
\!\!\ln  \left( \sqrt {{\frac{1}{n^2}}-{a}^{2}+{d}^{2}} \!- \!\sqrt {{\frac{1}{n^2}}-{a}^{2}}
 \right) \!- \!\ln  \left( \sqrt {{\frac{1}{n^2}}-{a}^{2}+{d}^{2}}+\sqrt {{\frac{1}{n^2}}-{a}^{2}} \right) \right)  +  \\
& & \left.\sqrt {{\frac{1}{n^2}}-{a}^{2}}  \left( \ln  \left( \sqrt {{\frac{1}{n^2}}-
{a}^{2}+{d}^{2}}-d \right) - \ln  \left( \sqrt {
{\frac{1}{n^2}}-{a}^{2}+{d}^{2}}+d \right) \right) \right).
\end{eqnarray*}
The double integral $B(x,y)$ is evaluated in a completely analogous fashion.

The integration domains of $F_2$ and $F_5$ have a similar geometric structure. We give a short overview of the evaluation of $F_2$. As in the case of $F_3$  we split the integration domain along the $y$-axis in order to have less case distinction when transforming to polar coordinates:
\[
  F_2(x,y) := \int^{\sqrt{\frac{1}{n^2} - d^2}}_{ \sqrt{\frac{1}{n^2} - a^2}}\! \int^{0}_{d}
\frac{1}{\abs{(w, z)}}  \; dw dz  + \int^{\sqrt{\frac{1}{n^2} - d^2}}_{ \sqrt{\frac{1}{n^2} - a^2}}\! \int^{ \sqrt{\frac{1}{n^2} - z^2}}_{0}
\frac{1}{\abs{(w, z)}}  \; dy  dx.
\]
The first double integral has an axis-parallel rectangle domain located in the second quadrant and is treated like $A(x,y)$. The domain of the second double integral is the intersection of two shapes: an axis-parallel  rectangle domain located in the first quadrant whose lower right vertice $B$ lies on the circle $B(0,\frac{1}{n})$, and the circle $B(0,\frac{1}{n})$ itself. This means that by the intersection the right edge of the rectangle and parts of its top edge are exchanged with a circular arc  around zero with radius $\frac{1}{n}$. The double integral is translated to polar coordinates as follows:
\[
 \int_{\sqrt{\frac{1}{n^2} - a^2}}^{\sqrt{\frac{1}{n^2} - d^2}}  \int_{\arcsin\left(\frac{\sqrt{\frac{1}{n^2} - a^2}   }{r} \right)}^\frac{\pi}{2} \; d\phi dr +
\int_{\sqrt{\frac{1}{n^2} - d^2}}^\frac{1}{n} \int_{\arcsin\left(\frac{\sqrt{\frac{1}{n^2} - a^2}   }{r} \right)}^{\arcsin\left(\frac{\sqrt{\frac{1}{n^2} - d^2}   }{r}\right)} \; d\phi dr.
\]

For the evaluation of $F_5$ proceed as for $F_2$.
The only difference to $F_2$  lies in the fact that by the intersection with the circle $B(0, \frac{1}{n})$ also parts of the bottom line of the corresponding underlying rectangle are removed. In symbols this is reflected by exchanging every occurrence of the term $\sqrt{\frac{1}{n^2} - a^2}$ with $b$ in the double integral directly above.

 The integration of the functions $F_i$ with respect to $y$ can be executed in all cases by standard means. Note that the double integrals related to $F_1$, $F_2$ and $F_4$ do not depend on $y$ such that those integrations are mere multiplications of the respective functions with the difference between the limits of the respective integrals. After all the resulting functions of the variable $x$ have been summed up the fourth integration can be carried out yielding
 \[
  J_\frac{1}{2} = \sum_{i=1}^6 K_i = \sum_{i=1}^6 L_i = \frac{1}{3n^3} \quad \mathrm{and} \quad J_{i,j}^{i,j-1} = 2 J_\frac{1}{2} = \frac{2}{3n^3}.
\]
This solves the case of laterally adjacent squares

The case of diagonally adjacent squares is much simpler. For $2 \leq i,j \leq n$ we have to compute
\[
 J_{i,j}^{i-1,j-1} = \int_{\frac{i-1}{n}}^\frac{i}{n} \! \int_{\frac{j-1}{n}}^\frac{j}{n} \! \int \!\!\!\int_{S_{i-1,j-1}\left(x,y,\frac{1}{n}\right)} \!\!
\frac{1}{\abs{(x- w, y -z)}}  \; d(w,z)\, dy  dx.
\]
where
\[
 S_{i-1,j-1}\left(x,y,\frac{1}{n}\right) = \left(\left(\frac{i-2}{n},\frac{i-1}{n}\right) \times \left(\frac{j-2}{n},\frac{j-1}{n} \right) \right)   \cap B\left((x,y), \frac{1}{n}\right).
\]

An easy computation shows that the latter set is empty if and only if $|(a,b)| \geq \frac{1}{n}$. Therefore, for the computation of the inner double integral of $J_{i,j}^{i-1,j-1}$ we may restrict ourselves to points $(x,y) \in I_{i,j}$   that satisfy $|(a,b)| < \frac{1}{n}$. By  similar reasoning as in the lateral case we infer that $J_{i,j}^{i-1,j-1}$ equals
\[
 \int_{\frac{i-1}{n}}^\frac{i}{n} \! \int_{\frac{j-1}{n}}^{\frac{j-1}{n}+\sqrt{\frac{1}{n^2} - a^2}} \! \int_{y- \sqrt{\frac{1}{n^2} - a^2}}^{\frac{j-1}{n}} \int_{x- \sqrt{\frac{1}{n^2} - \left(z-y\right)^2}}^{\frac{i-1}{n}}
\frac{1}{\abs{(x- w, y -z)}}  \; dw dz  dy  dx.
\]

By applying the transformation $(w,z) \mapsto (x-w, y-z)$ this quadruple integral transforms to
\[
 \int_{\frac{i-1}{n}}^\frac{i}{n} \! \int_{\frac{j-1}{n}}^{\frac{j-1}{n}+\sqrt{\frac{1}{n^2} - a^2}} \! \int^{\sqrt{\frac{1}{n^2} - a^2}}_{b} \int^{\sqrt{\frac{1}{n^2} - z^2}}_{a}
\frac{1}{\abs{(w, z)}}  \; dw dz  dy  dx,
\]
and by changing the inner double integral to polar coordinates we get
\begin{eqnarray*}
 J_{i,j}^{i-1,j-1} & = & \int_{\frac{i-1}{n}}^\frac{i}{n} \! \int_{\frac{j-1}{n}}^{\frac{j-1}{n}+\sqrt{\frac{1}{n^2} - a^2}} \! \int^{\frac{1}{n}}_{|(a,b)|} \int^{\arccos\left(\frac{a}{r}\right)}_{\arcsin\left(\frac{b}{r}\right)}
\frac{1}{\abs{(w, z)}}  \; dw dz  dy  dx \\
& = & \frac{1}{6 n^3}.
\end{eqnarray*}
Altogether, by (\ref{result}) the final result is
\[
 \Reg{n}{1}(f) = \frac{1}{3 \pi n} \left(\sum_{i=2}^n\sum_{j=2}^n\left|a_{i,j} - a_{i-1,j-1}\right| + \sum_{i=1}^{n-1} \sum_{j =2}^{n} \left|a_{i,j} - a_{i+1,j-1}\right| \right)  +
\]
\[
 \frac{4}{3\pi n} \left(\sum_{i=1}^n\sum_{j=2}^n \left|a_{i,j} - a_{i,j-1}\right| + \sum_{i=2}^{n} \sum_{j =1}^{n} \left|a_{i,j} - a_{i-1,j}\right| \right).
\]

\item
With $f$ as above  we also evaluated $\Reg{n}{1}(f)$ with the kernel functions
\[
 \varphi_n(\vec{v}) := \frac{n^2}{4} \chi_{(-\frac{1}{n}, \frac{1}{n}) \times (-\frac{1}{n}, \frac{1}{n})}(\vec{v})
\]
for all $\vec{v}$ in $\R^2$. Note that these kernel functions
are not radial such that this case is not treated in the theory section.
The evaluation proceeds similar to the above one. As a result we get
\[
 \Reg{n}{1}(f) = \frac{1}{3} \frac{\sqrt{2} -1}{n} \left( \sum_{i=2}^n\sum_{j=2}^n \left|a_{i,j} - a_{i-1,j-1}\right| + \sum_{i=1}^{n-1} \sum_{j =2}^{n} \left|a_{i,j} - a_{i+1,j-1}\right| \right)+
\]
\[
 \frac{1}{12}\,{\frac {3\,\ln  \left( \sqrt {2} +1\right) -3\,\ln  \left(
\sqrt {2}-1 \right) - 2\,(\sqrt {2}-1)}{n}} \times
\]
\[
\left( \sum_{i=1}^n\sum_{j=2}^n \left|a_{i,j} - a_{i,j-1}\right| + \sum_{i=2}^{n} \sum_{j =1}^{n} \left|a_{i,j} - a_{i-1,j}\right| \right).
\]
\end{enumerate}

\section*{Acknowledgement}
The authors would like to express their gratitude to Paul F. X. M\"uller for introducing us to the recnt work on the new characterizations of Sobolev spaces and BV and some stimulating discussions. 
This work has been supported by the Austrian Science Fund (FWF) within the national research networks Industrial Geometry,
project 9203-N12, and Photo\-acoustic Imaging in Biology and Medicine, project S10505-N20.

\IfFileExists{../Bibliographie/Users_Guide.txt}{
  \def\BibPath{../Bibliographie/}}{
  \IfFileExists{../../Bibliographie/Users_Guide.txt}{
    \def\BibPath{../../Bibliographie/}}{
    \IfFileExists{../../../Bibliographie/Users_Guide.txt}{
      \def\BibPath{../../../Bibliographie/}}{
        \IfFileExists{../../../../Bibliographie/Users_Guide.txt}{
          \def\BibPath{../../../../Bibliographie/}}{
          \logmessage{Directory 'Bibliographie' not found}
          \def\BibPath{}}}}}

\bibliographystyle{\BibPath hyperref_plain}
\bibliography{\BibPath books,\BibPath proceedings,\BibPath articles,\BibPath accepted,\BibPath infmath,\BibPath inproceedings,\BibPath preprints,\BibPath theses,\BibPath websites}

\end{document}